\documentclass[12pt]{amsart}
\usepackage{Preamble}

\title{Tropical Moduli Space of Rational Graphically Stable Curves}
\author{Andy Fry}

\begin{document}




\begin{abstract}
We study moduli spaces of rational graphically stable tropical curves and a refinement given by radial alignment. Given a complete multipartite graph $\Gamma$, the moduli space of radially aligned $\Gamma$-stable tropical curves can be given the structure of a balanced fan. This fan structure coincides with the Bergman fan of the cycle matroid of $\Gamma$.
\end{abstract}

\maketitle

\section{Introduction}

The moduli space $\Mzerontrop$ is a cone complex which parameterizes leaf-labelled metric trees. Its structure is obtained by gluing positive orthants of $\RR^{n-3}$ corresponding to trivalent trees. Speyer and Sturmfels \cite{speyer2004tropical} give an embedding of this cone complex (in the context of phylogenetic trees) into a real vector space as a balanced fan where each top-dimensional cone is assigned weight $1$. In \cite{ardila2006bergman}, Ardila and Klivans study phylogenetic trees and show that the fan structure of $\Mzerontrop$ has a refinement which coincides with the Bergman fan of the cycle matroid of $K_{n-1}$, the complete graph on $n-1$ vertices. As a generalization of Ardila and Klivans, it is shown by Cavalieri, Hampe, Markwig, and Ranganathan in \cite{cavalieri2016moduli} that the fan associated to the moduli space of rational heavy/light weighted stable tropical curves, $\M_{0,w}^{\textrm{trop}}$, and the Bergman fan of a graphic matroid have the same support. 

The main result of this paper further generalizes their result by starting with stability conditions defined by a graph (reduced weight graph) rather than a weight vector.
We introduce \emph{tropical rational graphically stable curves} 
(Definition~\ref{def:GammaStability}) and write $\MzeroGtrop$ as the moduli space of these curves.
We define this moduli space so that if we begin with a graph that is also a reduced weight graph (Definition 2.13 of \cite{cavalieri2016moduli}) we recover the corresponding weighted moduli space.

Keeping in mind that the Bergman fan of a graphic matroid is a refinement of $\Mzerontrop$ (also $\M_{0,w}^{\textrm{trop}}$), we add the extra condition of \emph{radial alignment} to $\Mzerontrop$ to define the moduli space $\Mzeronrad$. 
Radial alignment refers to an ordered partition on the vertices of the combinatorial type of a curve. 
One result is an independent restatement of results from \cite{ardila2006bergman} and \cite{franccois2013diagonal}. 
It shows that the radial alignment is the condition needed to achieve the Bergman fan refinement.\\

\noindent\textbf{Lemma~\ref{lem:MzeronradIsEqualToBKprimeAsConeComplexes}} The moduli space $\Mzeronrad$ has a fan structure equal to the Bergman fan of the cycle matroid of $K_{n-1}$. \\

Our main result characterizes tropical matroidal moduli spaces.\\

\noindent\textbf{Theorem~\ref{thm:BergmanFanEqualsMzeroGammarad}} The moduli space $\MzeroGrad$ has the structure of a balanced fan if and only if $\Gamma$ is a complete multipartite graph. Furthermore, this fan is equal to the Bergman fan of the cycle matroid of $\Gamma$.\\

Our motivation for this paper comes from the theory of tropical compactifications and log-geometry.
From work of Tevelev \cite{tevelev2007compactifications} and Gibney and Maclagan \cite{gibney2011equations} it has been shown that there is an embedding of $\M_{0,n}$ into a toric variety $X(\Sigma)$ where the tropicalization of $\M_{0,n}$ is a balanced fan $\Sigma\cong\Mzerontrop$. This embedding is special in the sense that the closure of $\M_{0,n}$ in $X(\Sigma)$ is $\overline{\M}_{0,n}$. Cavalieri et al. \cite{cavalieri2016moduli} show a similar embedding can be constructed for weighted moduli spaces when the weights are heavy/light spaces. We begin this process for matroidal spaces. 
In \cite{ranganathan2017moduli} Ranganathan, Santos-Parker, and Wise describe radial alignments of genus 1 tropical curves and show how this extra data can be used for desingularization. \\

The paper is organized as follows. In section 2 we define a matroid axiomatically using independence axioms. Afterwards we describe how a matroid can also be defined via base, rank, closure, and circuit axioms. Then we restrict our attention to the cycle matroid and discuss necessary graph theory terminology.

In Section 3 we begin by defining the moduli space of rational $n$-marked tropical curves, $\Mzerontrop$. We also describe an embedding as a balanced fan into a real vector space as in \cite{gathmann2009tropical}. 
In Section 3.1, we define the Bergman fan of a matroid and also define radially aligned rational $n$-marked tropical curves by imposing a weak ordering on the vertices given by their distances from the root vertex. Using radially aligned tropical curves we describe an explicit constructive bijection between cones of the $\BKprime$ and cones of $\Mzeronrad$.

Section 3.2 is original work motivated by \cite{cavalieri2016moduli}. We define graphically stable radially aligned tropical curves. In subsection 3.2.1 we investigate projections of $\Mzeronrad$ and $\BKprime$ by forgetting coordinates of rays corresponding unstable curves and show that the fans coincide with $\BGammaprime$. Subsection 3.2.2 contains our final result that states $\MzeroGrad$ can be embedded as a balanced fan equal to the $\BGammaprime$ when $\Gamma$ is a complete $k$-partite graph. \\

      
\section{Matroids}

\subsection{Introduction to Matroids}

The concept of a matroid was independently discovered in 1930's by Whitney \cite{whitney1935abstract}, van der Waerden \cite{van1950moderne}, and Nakasawa \cite{nishimura2009lost}. Whitney's original paper looked at the similarities between linear independence and graph theoretic independence. Similarly, van der Waerden was also interested in generalizing the notion of independence by comparing linear independence and algebraic independence. 

Over the next 30 years the following key results arose. In the 30's Birkhoff made the connection that one of the rank axioms ((R3$'$) specifically) is the semimodular condition for lattices \cite{birkhoff1935abstract} and Mac Lane wrote an article on the relations to projective geometry \cite{maclane1936some}. The 1940's saw expansions by Rado with work on transversality \cite{rado1942theorem} and infinite matroids \cite{rado1949axiomatic} and Dilworth who wrote more on lattice theory \cite{dilworth1944dependence}. It wasn't until the late 50's/early 60's when matroid theory took off. Much of this due to the results of Rado, Tutte, Edmonds, and Lehman. Highlighting some results of Tutte are the categorization of binary and regular \cite{tutte1958homotopy}, and graphic \cite{tutte1959matroids} matroids.

Since then matroids have been a target study for linear algebra, graph theory, optimization, block designs, combinatorial algebraic geometry and more. We begin as Whitney did, the axiomatic definition in terms of independence.\\

A \emph{matroid} is a tuple $M=(E,I)$ where $E$ is a finite set (called the \emph{ground set}) and $I$ is a collection of subsets of $E$ such that \ref{ax:I1}--\ref{ax:I3} are satisfied.

\begin{enumerate}[label=(\textbf{I\arabic*})]
\item \label{ax:I1} $\emptyset\in I$
\item \label{ax:I2} If $X\in I$ and $Y\subseteq X$, then $Y\in I$
\item \label{ax:I3} If $U,V\in I$ with $|U|=|V|+1$, then there exists $x\in U\setminus V$ such that $V\cup x\in I$.
\end{enumerate}
The elements of $I$ are called \emph{independent sets} and thusly call \ref{ax:I1}, \ref{ax:I2}, and \ref{ax:I3} the \emph{independence axioms}. If a subset of $E$ is not independent, then we call it \emph{dependent}. More commonly \ref{ax:I3} is known as the \emph{exchange property}. Other resources tend to restrict the definition of a matroid to just the latter two properties.\\

For a matroid $M=(E,I)$ we make the following definitions. A \emph{base} $B$ of $M$ is a maximal independent subset of $E$. A \emph{circuit} $C$ of $M$ is a minimal dependent set. Minimal and maximal refer to the size of the circuit or base. Denote $2^E$ as the power set of $E$. The \emph{rank function} of a matroid is $\rk:2^E\rightarrow\ZZ$ defined by 
$$\rk(A)=\max(|X|:X\subseteq A,X\in I).$$
In the case where $A\in I$, then $\rk(A)=|A|$. 

We are also interested in the notion of a flat or subspace. A subset $F\subseteq E$ is a \emph{flat} (also called a \emph{subspace} or \emph{closed}) of $M(\Gamma)$ if for all $x\in E\setminus F$, 
    $$\rk(F\cup x)=\rk(F)+1.$$
In other words,
$F$ is a flat if there are no elements that can be added to $F$ without increasing the rank of $F$. 
Define the \emph{closure operator} to be a function $\cl:2^E\rightarrow2^E$ such that $\cl(A)$ is the set of elements that satisfy the following property. If $x\in E$ and $A\subset E$, then $\cl(A\cup x)=\cl(A)$. It turns out that $\cl(A)$ is the smallest flat containing $A$. Ordered by rank we may associate a partially ordered set (poset) to the flats of a matroid. It turns out that this poset actually forms a lattice, called the \emph{lattice of flats}.\\

Next we give some equivalent axiomatic definitions of a matroid as presented by Welsh in \cite{welsh1976matroid}, the first comprehensive book on matroid theory.\\

\noindent{\bf Base Axiom:} A non-empty collection $\mathscr{B}$ of subsets of $E$ is the set of bases of a matroid on $E$ iff it satisfies the following condition:

\begin{enumerate}[label=(\textbf{B\arabic*})]
    \item \label{ax:B1} For $B_1, B_2\in\mathscr{B}$ and for $x\in B_1\setminus B_2$, there exists $y\in B_2\setminus B_1$ such that $(B_1\cup y)\setminus x \in\mathscr{B}.$\\
\end{enumerate}

There are two ways to define a matroid in terms of rank.\\

\noindent{\bf Rank Axioms 1:} A function $\rk:2^E\rightarrow\ZZ$ is the rank function of a matroid $E$ if and only if for $X\subseteq E,$ and $y,z\in E$:

\begin{enumerate}[label=(\textbf{R\arabic*})]
    \item \label{ax:R1} $\rk(\emptyset)=0$;
    \item \label{ax:R2} $\rk(X)\leq\rk(X\cup y)\leq\rk(X)+1$;
    \item \label{ax:R3} if $\rk(X\cup y)=\rk(X\cup z)=\rk(X)$ then $\rk(X\cup y \cup z)=\rk(X).$\\
\end{enumerate}

\noindent{\bf Rank Axioms 2:} A function $\rk:2^E\rightarrow\ZZ$ is the rank function of a matroid $E$ if and only if for any subsets $X$, $Y$ of $E$:

\begin{enumerate}[label=(\textbf{R\arabic*}$'$)]
    \item \label{ax:R1'} $0\leq\rk(X)\leq|X|$;
    \item \label{ax:R2'} $X\subseteq Y\Rightarrow \rk(X)\leq\rk(Y)$;
    \item \label{ax:R3'} $\rk(X\cup Y)+\rk(X\cap Y)\leq \rk(X)+\rk(Y).$ \\
\end{enumerate}

\noindent{\bf Closure Axioms:} A function $\cl:2^E\rightarrow2^E$ is the closure operator of a matroid on $E$ if and only if for $X,Y\subset E$ and $x,y\in E$:

\begin{enumerate}[label=(\textbf{S\arabic*})]
    \item \label{ax:S1} $X\subseteq\cl(X)$
    \item \label{ax:S2} $Y\subseteq X \Rightarrow \cl(Y)\subseteq\cl(X)$ 
    \item \label{ax:S3} $\cl(X)=\cl(\cl(X))$
    \item \label{ax:S4} if $y\not\in\cl(X)$ but $y\in\cl(X\cup x)$, then $x\in\cl(X\cup y)$.\\
\end{enumerate}

\noindent{\bf Circuit Axioms:} A collection $\mathscr{C}$ of subsets of $E$ is the set of circuits of a matroid on $E$ if and only if the following conditions are satisfied:

\begin{enumerate}[label=(\textbf{C\arabic*})]
    \item \label{ax:C1} If $X\neq Y\in\mathscr{C},$ then $X\not\subseteq Y.$
    \item \label{ax:C2} If $C_1, C_2$ are distinct members of $\mathscr{C}$ and $z\in C_1\cap C_2$ then there exists $C_3\in\mathscr{C}$ such that $C_3\subseteq(C_1\cup C_2)\setminus z.$
\end{enumerate}

\subsection{The Cycle Matroid}




In this paper, we consider the matroid of a finite simple connected graph $\Gamma=(V,E)$ where $V$ is the ordered vertex set and $E$ the edge set, sometimes denoted $E(\Gamma)$. We denote $e_{ij}\in E$ to be an edge between vertices $v_i$ and $v_j$. A graph is \emph{complete} if each pair of distinct vertices has an edge between them. The complete graph on $n$ vertices is denoted $K_n$. A \emph{clique} is a subgraph that is complete, denoted $K_I$ where $I$ is the set of vertices with edges between them. A disjoint union of complete graphs is called a \emph{cluster graph}. 

Often called the \emph{cycle matroid}, this matroid is given by $M(\Gamma)=(E(\Gamma),I)$ where $I$ is the collection of all forests of $\Gamma$. A \emph{forest} of $\Gamma$ is a subgraph, possibly disconnected, such that all connected components are trees.



\begin{lem}
Let $\Gamma$ be as above. Then $M(\Gamma)$ is a matroid.
\end{lem}
\begin{proof}
\begin{enumerate}
    \item[(I1)] The graph with no edges is a forest since it contains no cycle.
    \item[(I2)] A subgraph of a forest is a forest
    \item[(I3)] Let $A$ and $B$ be forests with $|A|=|B|+1$. Let $k_A$, $V_A$ and $k_B$, $V_B$ denote the number of connected components and vertex sets of $A$ and $B$, respectively. It is known that $|A|=|V_A|-k_A$, similarly for $B$. Thus $|V_A|-k_A=|V_B|-k_B+1$.\\
    If $V_A\setminus V_B\neq\emptyset$, pick $e\in A$ such that one of its ends is in  $V_A\setminus V_B$. Then $B\cup \{e\}$ is a forest. However, if $V_A\setminus V_B=\emptyset$, then $V_A\subseteq V_B$. Therefore there must be an edge $e\in A$ that connects two components of $B$ otherwise $|A|\leq|B|$. Again $e$ suffices our conclusion.
    

    
    
    
\end{enumerate}
\end{proof}

Consider the cycle matroid $M(\Gamma)$. A \emph{base} is a spanning forest. A \emph{circuit} is a cycle, which is a path in which the initial and terminal vertices are the same and no other vertices repeat. The \emph{rank} of a set of edges $E'$ is the number of edges in a spanning forest of $\Gamma_{E'}$, the subgraph induced by $E'$. Alternatively, the rank of a subgraph $G\subset\Gamma$ may be computed by $n-k$ where $n$ is the number of non-isolated vertices in $G$ and $k$ is the connected components of among non-isolated vertices of $G$. An \emph{isolated vertex} is a vertex that is not a part of an edge. \\


We restrict our attention to $\Gamma=K_n$ to examine flats and the closure operator. A \emph{flat} of $M(K_n)$ is a cluster graph, $\coprod_{j=1}^{k}K_{I_j}$. For a subgraph $G$ of $K_n$, whose connected components are given by vertex sets $V_1,\ldots,V_k$. Then $\cl(G)$ is the flat $\coprod_{j=1}^{k}K_{V_j}$. We say that the closure operator is \emph{completing} each connected component. See Figure~\ref{fig:LatticeOfFlatsOfK4LabeledByEdges} for the lattice of flats for $M(K_4)$. 


\begin{figure}[h]
\begin{subfigure}{.39\textwidth}
    \centering
\begin{tikzpicture}[scale=2]
    \node (v2) at (1,1.732) [circle, draw = black, inner sep =  2pt, outer sep = 0.5pt, minimum size = 4mm, line width = 1pt] {$2$};
    \node (v3) at (2,0) [circle, draw = black, inner sep =  2pt, outer sep = 0.5pt, minimum size = 4mm, line width = 1pt] {$3$};
    \node (v4) at (0,0) [circle, draw = black, inner sep =  2pt, outer sep = 0.5pt, minimum size = 4mm, line width = 1pt] {$4$};
    \node (v5) at (1,0.577) [circle, draw = black, inner sep =  2pt, outer sep = 0.5pt, minimum size = 4mm, line width = 1pt] {$5$};
    
    \draw[line width = 1pt] (v2) -- node [midway, above right] {$e_{23}$} (v3);
    \draw[line width = 1pt] (v2) -- node [midway, left] {$e_{24}$} (v4);
    \draw[line width = 1pt] (v2) -- node [midway, below left] {$e_{25}$} (v5);
    \draw[line width = 1pt] (v3) -- node [midway, below] {$e_{34}$} (v4);
    \draw[line width = 1pt] (v3) -- node [midway, above] {$e_{35}$} (v5);
    \draw[line width = 1pt] (v4) -- node [midway, right] {$e_{45}$} (v5);
\end{tikzpicture}
    \caption{$K_4$}
    \label{fig:LabeledK4}
\end{subfigure}
\hfill
\begin{subfigure}{.59\textwidth}
\begin{tikzpicture}[scale=.75]
\draw[fill=black] (0,0) circle (3pt);
\draw[fill=black] (-5,3) circle (3pt);
\draw[fill=black] (-3,3) circle (3pt);
\draw[fill=black] (-1,3) circle (3pt);
\draw[fill=black] (1,3) circle (3pt);
\draw[fill=black] (3,3) circle (3pt);
\draw[fill=black] (5,3) circle (3pt);
\draw[fill=black] (-6,6) circle (3pt);
\draw[fill=black] (-4,6) circle (3pt);
\draw[fill=black] (-2,6) circle (3pt);
\draw[fill=black] (0,6) circle (3pt);
\draw[fill=black] (2,6) circle (3pt);
\draw[fill=black] (4,6) circle (3pt);
\draw[fill=black] (6,6) circle (3pt);
\draw[fill=black] (0,9) circle (3pt);

\draw[thick] (0,0) -- (-5,3);
\draw[thick] (0,0) -- (-3,3);
\draw[thick] (0,0) -- (-1,3);
\draw[thick] (0,0) -- (1,3);
\draw[thick] (0,0) -- (3,3);
\draw[thick] (0,0) -- (5,3);

\draw[thick] (-5,3) -- (-6,6);
\draw[thick] (-5,3) -- (-4,6);
\draw[thick] (-5,3) -- (2,6);

\draw[thick] (-3,3) -- (-6,6);
\draw[thick] (-3,3) -- (0,6);
\draw[thick] (-3,3) -- (4,6);

\draw[thick] (-1,3) -- (-6,6);
\draw[thick] (-1,3) -- (-2,6);
\draw[thick] (-1,3) -- (6,6);

\draw[thick] (1,3) -- (-4,6);
\draw[thick] (1,3) -- (-2,6);
\draw[thick] (1,3) -- (4,6);

\draw[thick] (3,3) -- (-2,6);
\draw[thick] (3,3) -- (0,6);
\draw[thick] (3,3) -- (2,6);

\draw[thick] (5,3) -- (-4,6);
\draw[thick] (5,3) -- (0,6);
\draw[thick] (5,3) -- (6,6);

\draw[thick] (0,9) -- (-6,6);
\draw[thick] (0,9) -- (-4,6);
\draw[thick] (0,9) -- (-2,6);
\draw[thick] (0,9) -- (0,6);
\draw[thick] (0,9) -- (2,6);
\draw[thick] (0,9) -- (4,6);
\draw[thick] (0,9) -- (6,6);
\end{tikzpicture}
    \caption{Lattice of flats of $M(K_4)$}
    \label{fig:LatticeOfFlatsOfK4LabeledByEdges}
\end{subfigure}
\caption{}
\end{figure}

One natural operation on a graph is to delete edges. The cycle matroid respects this operation in the sense that a subgraph induces a submatroid, called the \emph{restriction matroid}. Note that the restriction of an arbitrary matroid is a well-defined concept in general but in this paper we only use it in the context of the cycle matroid of $K_n$.

\begin{lem}
(Theorem 1 from \cite{welsh1976matroid}, section 4.2) 
Let $G$ be a subgraph of $K_n$ and denote $M(K_n)=(E,I)$. Let $I|G=\{X | X\subseteq E(G) \textrm{ and } X\in I \}$ be the restriction of forests of $K_n$ to the edge set of $G$. Then $I|G$ is the set of independent sets of the $M(G)$.
\end{lem}

That is, the forests of $G$ can be obtained by intersecting a forest of $\Gamma$ with $G$. Denote the closure operators for $M(K_n)$ and $M(G)$ as $\cl_{K_n}$ and $\cl_G$, respectively. They are related by 
\begin{align}\label{eqn:FlatsOfSubgraphAreRestrictionsOfFlatsOfGraph}
    \cl_G(A)=\cl_{K_n}(A)\cap G.
\end{align}
Unlike the closure operator, there is no ambiguity between the rank functions on $M(G)$ and $M(K_n)$ so we will denote both as $\rk(A)$.\\

Suppose the graph $G$ is obtained by removing edge $e_{25}$ from $K_4$, as labeled in Figure~\ref{fig:LabeledK4}. Analyzing its lattice of flats, we see that we obtain a sublattice of the lattice of flats of $M(K_4)$, see Figure~\ref{fig:LatticeOfFlatsOfK5MinusAnEdge}.\\


\begin{figure}[h]
    \centering
\begin{tikzpicture}[scale=.75]
\draw[fill=black] (0,0) circle (3pt);
\draw[fill=black] (-4,2) circle (3pt);
\draw[fill=black] (-2,2) circle (3pt);
\draw[fill=black] (0,2) circle (3pt);
\draw[fill=black] (2,2) circle (3pt);
\draw[fill=black] (4,2) circle (3pt);
\draw[fill=black] (-5,4) circle (3pt);
\draw[fill=black] (-3,4) circle (3pt);
\draw[fill=black] (-1,4) circle (3pt);
\draw[fill=black] (1,4) circle (3pt);
\draw[fill=black] (3,4) circle (3pt);
\draw[fill=black] (5,4) circle (3pt);
\draw[fill=black] (0,6) circle (3pt);

\draw[thick] (0,6) -- (-5,4);
\draw[thick] (0,6) -- (-3,4);
\draw[thick] (0,6) -- (-1,4);
\draw[thick] (0,6) -- (1,4);
\draw[thick] (0,6) -- (3,4);
\draw[thick] (0,6) -- (5,4);

\draw[thick] (-5,4) -- (-4,2);
\draw[thick] (-5,4) -- (-2,2);
\draw[thick] (-5,4) -- (0,2);

\draw[thick] (-3,4) -- (-4,2);
\draw[thick] (-3,4) -- (2,2);

\draw[thick] (-1,4) -- (0,2);
\draw[thick] (-1,4) -- (2,2);
\draw[thick] (-1,4) -- (4,2);

\draw[thick] (1,4) -- (-2,2);
\draw[thick] (1,4) -- (4,2);

\draw[thick] (3,4) -- (-4,2);
\draw[thick] (3,4) -- (4,2);

\draw[thick] (5,4) -- (-2,2);
\draw[thick] (5,4) -- (2,2);

\draw[thick] (0,0) -- (-4,2);
\draw[thick] (0,0) -- (-2,2);
\draw[thick] (0,0) -- (0,2);
\draw[thick] (0,0) -- (2,2);
\draw[thick] (0,0) -- (4,2);
\end{tikzpicture}
    \caption{Lattice of flats of $M(\Gamma)$}
    \label{fig:LatticeOfFlatsOfK5MinusAnEdge}
\end{figure}
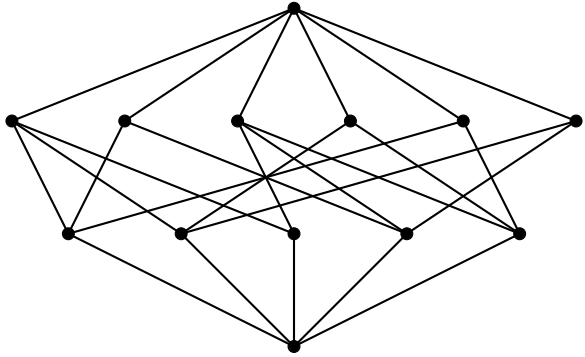

Next we write a technical lemma used for Proposition~\ref{prop:CriterionForInjectivityOfprojg}. The proof is purely graph theoretic so we prove it here.

\begin{lem}\label{lem:RankOfASubgraphIFFCommonSpanningForest}
Let $\Gamma$ be a simple graph, not necessarily connected and let $G$ be a subgraph of $\Gamma$. Then $\rk(\Gamma)=\rk(G)$ if and only if $G$ and $\Gamma$ share a common spanning forest.
\end{lem}
\begin{proof}
The backwards direction follows from the definition of rank. So let us assume that $G$ and $\Gamma$ have the same rank. Let $T'$ be a spanning forest of $G$. Then there exists $T$ a spanning forest of $\Gamma$ such that $T\cap G=T'$. By assumption we know that $\rk(T)=\rk(T')$ and therefore they have the same number of edges. Since $T'$ is a subgraph of $T$, they must be the equal.
\end{proof}

Here we define the complete multipartite graph and discuss some facts about it. A \emph{$k$-partite graph} (or multipartite) is a graph on $n=\sum_{i=1}^k n_i$ vertices, partitioned into $k$ sets (called \emph{independent sets}) such that no two vertices from the same set are adjacent. A \emph{complete} $k$-partite graph denoted $K_{n_1,\ldots,n_k}$ is a $k$-partite graph such that every pair of vertices in different sets are adjacent. Alternatively, we may obtain a complete $k$-partite graph by removing the disjoint cliques on vertices given by the independent sets. Meaning that the complement of a complete mutlipartite graph is a cluster graph.\\

The following lemma describes some useful characterizations of a complete multipartite graph. The proof follows from graph theoretic properties.

\begin{lem}\label{lem:AlternateDescriptionsOfCompleteMultipartiteGraphs}
Let $G$ be a graph. The following are equivalent:
\begin{enumerate}
    \item $G$ is a complete multipartite graph.
    \item If $e_{ij}$ is an edge of $G$, then for any other vertex $v_k$ either $e_{ik}$ or $e_{jk}$ is an edge of $G$.
    \item There do not exist 3 vertices whose induced subgraph has exactly 1 edge. 
\end{enumerate}
\end{lem}



      
\section{Tropical Moduli Spaces as Bergman Fans}




Consider the space of genus 0, $n$-marked abstract tropical curves $\Mzerontrop$. That is, points of $\Mzerontrop$, denoted $\C$, are in bijection with metrized trees with \emph{bounded edges} having finite length and $n$ unbounded labeled edges called \emph{ends}. By forgetting the lengths of the bounded edges of $\C$ we get a tree with labeled ends called the \emph{combinatorial type} of $\C$. 

We may also think of $\Mzerontrop$ as a cone complex. Curves of a fixed combinatorial type with $d$ bounded edges are parameterized by $\RR_{\geq0}^{d}$. We obtain $\Mzerontrop$ by gluing several copies of $\RR_{\geq0}^{n-3}$ via appropriate face morphisms, one for each trivalent combinatorial type. 
Note that a $(d-m)$-dimensional face of a $d$-dimensional cone corresponds to a combinatorial type where the lengths of $m$ bounded edges are shrunk to 0.

Furthermore we discuss the embedding of $\Mzerontrop$ into a real vector space as a balanced, weighted, pure-dimensional polyhedral fan as in \cite{gathmann2009tropical}. A \emph{weighted fan} $(X,\omega)$ is a fan $X$ in $\RR^n$ where each cone $\sigma$ has a positive integer weight associated to it, denoted $\omega(\sigma)$. A weighted fan is \emph{balanced} if for all cones $\tau$ of codimension one, the weighted sum of primitive normal vectors of the top-dimensional cones $\sigma_i\supset\tau$ is 0, i.e.
$$\sum_{\sigma_i\supset\tau}\omega(\sigma_i)\cdot u_{\sigma_i/\tau}=0\in V/V_\tau$$
where $u_{\sigma_i/\tau}$ is the primitive normal vector, $V$ is the ambient real vector space, and $V_\tau$ is the smallest vector space containing the cone $\tau$. See \emph{Construction 2.3} of \cite{gathmann2009tropical} for a construction of the primitive normal vector. 

Denote $\dist(i,j)$ as the sum of lengths of all bounded edges between the ends marked by $i$ and $j$. Then vector
$$d(\C)=(\dist(i,j))_{i<j}\in\RR^{\binom{n}{2}}/\Phi(\RR^n)=Q_n$$
identifies $\C$ uniquely, where $\Phi:\RR^n\rightarrow\RR^{\binom{n}{2}}$ by $x\mapsto(x_i+x_j)_{i<j}$.

The combinatorial type of an abstract $n$-marked tropical curve $\C$ with one bounded edge \emph{splits} the set of ends $[n]$ into $I\sqcup I^c$. We denote the ray corresponding to $\C$ by $d(\C)=\rho_I=\rho_{I^c}$. We adopt the convention that in a split the set $I$ will not contain $1$. In \cite{kerber2009intersecting}, Kerber and Markwig proved the relation 
\begin{align}
    \sum_{S\in V_1}\rho_S=0\in Q_n \label{eqn:SumOfDivisorRaysEqualsZero}
\end{align} 
where $V_1=\{I|1\not\in I, |I|=2\}$. Also they showed that for $I\subset[n]\setminus\{1\}$ 
\begin{align}
\sum_{S\in\binom{I}{2}}\rho_S=\rho_I+\Phi(x)\in\RR^{\binom{n}{2}}\label{eqn:AnyRayIsASumOfDivisorRays}
\end{align}
where $\binom{I}{2}$ is the set of all size-2 subsets of a set $I$ and $x\in\RR^n$.

\begin{rem}\label{rem:BasisOfRaysEqualsDivisorCombTypes}
In other words, Equation~\eqref{eqn:SumOfDivisorRaysEqualsZero} means that any set of $\binom{n-1}{2}-1$ of combinatorial types of curves with one bounded edge and a trivalent vertex not containing the end $1$ corresponds to a basis of $Q_n$. Equation~\eqref{eqn:AnyRayIsASumOfDivisorRays} gives us the unique way to write any ray of $Q_n$ as a linear combination of our basis. 
\end{rem}

Now consider the combinatorial type of a tropical curve $\C$ with $d$ bounded edges. We construct $d$ splits, $I_1,\ldots,I_{d}$, in the following way. A split $I_j$ is defined by the combinatorial type you obtain contracting all but the $j$th bounded edge of $\C$. The cone corresponding to the combinatorial type of $\C$ is the span of rays $\rho_{I_1},\ldots,\rho_{I_d}$.\\


\subsection{Tropical moduli spaces of rational stable curves as Bergman fans}

Given any matroid $M$ with ground set $E$ we define the \emph{Bergman fan} which is a polyhedral fan $\BM\subseteq\RR^{|E|}$. The Bergman fan is 
$$ \BM:=\{ w\in\RR^{|E|} ~|~ M_w \textrm{ is loop-free}\},$$
where $M_w$ is the matroid on $E$ whose bases are all bases $B$ of $M$ which have minimal $w$-weight $\Sigma_{i\in B}w_i$. A \emph{loop} of a matroid is an element whose rank is 0.  A more useful definition for our purposes is from Ardila and Klivans \cite{ardila2006bergman}. They showed that $\BM$ is a polyhedral cone complex that coincides with the order complex of the lattice of flats of $M$. An \emph{order complex} of a poset $P$ is defined to be the simplicial complex whose vertices are the elements of $P$ and whose faces are chains of elements of $P$. 

In other words, given a chain of flats (COF) $\F$ in $M$
$$\emptyset\subsetneq F_1\subsetneq \cdots\subsetneq F_{r}\subsetneq F_{r+1}=E, $$
we let $C_\F$ be the cone in $\RR^{|E|}$ spanned by the rays $\rho_{F_1},\ldots,\rho_{F_{r+1}},$ with lineality space spanned by $\rho_E$. Here $\rho_F=-\Sigma_{e\in F}v_e$, where $v_e$ is a standard basis vector of $\RR^{|E|}$. 

\begin{rem}
Any Bergman fan contains the vector $(1,1,\ldots,1)$ as a ray. So rather than studying $\BM$ we quotient out the lineality space $L$, spanned by the vector $(1,1,\ldots,1)$, to get 
$$\BMprime:=\BM/L.$$
Thus we identify a COF $\F$ by its nontrivial flats $F_1,\ldots,F_r$ and denote $r$ to be its \emph{length}. Note that a COF of length $r$ corresponds to a cone of dimension $r$ in the Bergman fan. We call this polyhedral structure the \emph{chains-of-flats} subdivision of $\BMprime$.
\end{rem}


 
\begin{nota}
For the rest of the paper we write $\Gamma$ to represent the graph and the cycle matroid of $\Gamma$. We continue to use $|E(\Gamma)|$ as the number of edges in $\Gamma$.
\end{nota}

It has been shown in Section 4 of \cite{ardila2006bergman} and Example 7.2 of \cite{franccois2013diagonal} that the supports of $\Mzerontrop$ and $\BKprime$  coincide. Ardila and Klivans describe a bijection in terms of \emph{equidistant $(n-1)$-trees} while Fran\c{c}ois and Rau give a bijection induced by a $\ZZ$-linear transformation of the ambient spaces. We give an explicit constructive bijection between their stratifying sets. However, we must refine the notion of tropical curves to radially aligned tropical curves. These curves are analogous to the equidistant trees of Ardila and Klivans.\\

We define the \emph{root vertex} of a tropical curve $\C$ to be the vertex containing the end with marking 1 and we denote it $\V_0$. Given a labeling of the non-root vertices of $\C$, $\V_1,\ldots,\V_d$, we define $\ell_{i}$ to be the distance from the root vertex to $\V_i$. Also we set $\ell_0=0$.

\begin{defn}
A \emph{radially aligned} tropical curve $\C$ is a tropical curve with the additional data of a weak ordering on the vertices given by $\{\ell_{i}\}_{i=0}^{d}$. Define $\Mzeronrad$ as the parameter space of genus 0, $n$-marked radially aligned abstract tropical curves. Similar to before, we get the \emph{radially aligned combinatorial type} by forgetting the lengths but keeping the weak ordering on the vertices.
\end{defn}

\begin{rem}
A \emph{weak ordering} of a set can be viewed as an ordered partition. Meaning a partition of the vertices into disjoint subsets together with a total ordering on the subsets. Thus the number of cones of $\Mzeronrad$ can be counted using ordered Bell numbers or Fubini numbers. We can see this fact highlighted in Example~\ref{exam:Mzero7radFubini}. 
\end{rem}

Clearly the support of $\Mzerontrop$ and $\Mzeronrad$ are the same. But as fans, $\Mzeronrad$ is a refinement of $\Mzerontrop$. We call this refinement the \emph{radially aligned subdivision}. The next two examples illustrate particular 3-dimensional cones of $\Mzerontrop$ that become subdivided in the radially aligned subdivision.

\begin{exam}\label{exam:Mzero6rad}
Consider the combinatorial type $\C\in\M_{0,6}^{\textrm{trop}}$ with splits $I_1=\{2,3\},~I_2=\{4,5,6\},~I_3=\{5,6\}$, see Figure~\ref{fig:CombTypeOfMzero6rad}. In $\M_{0,6}^{\textrm{trop}}$, this combinatorial type corresponds to a single 3-dimensional cone with faces consisting of three 2-dimensional cones, and three rays. The 2-dimensional faces correspond to the combinatorial types obtained by shrinking the length of a bounded edge to 0. The rays correspond to contracting 2 bounded edges. In $\M_{0,6}^{\textrm{rad}}$, the radially aligned subdivision yields three distinct isomorphism classes, i.e. three 3-dimensional cones. By contracting the various bounded edges it turns out that there are seven 2-dimensional cones and five rays see Figure~\ref{fig:ConeOfMzero6rad}. The weak orderings are compiled in the 15 strings of inequalities listed below. 

\begin{minipage}{.3\textwidth}
\begin{align*}
    0=\ell_1<\ell_2=\ell_3\\
    0<\ell_1=\ell_2=\ell_3\\
    0=\ell_2=\ell_3<\ell_1\\
    0=\ell_2<\ell_1=\ell_3\\
    0=\ell_2=\ell_1<\ell_3
\end{align*} 
\end{minipage}
\begin{minipage}{.3\textwidth}
\begin{align*}
    0=\ell_1<\ell_2<\ell_3\\
    0<\ell_1<\ell_2=\ell_3\\
    0<\ell_2=\ell_3<\ell_1\\
    0=\ell_2<\ell_3<\ell_1\\
    0=\ell_2<\ell_1<\ell_3\\
    0<\ell_2=\ell_1<\ell_3\\
    0<\ell_2<\ell_1=\ell_3
\end{align*}
\end{minipage}
\begin{minipage}{.3\textwidth}
\begin{align*}
    0<\ell_1<\ell_2<\ell_3\\
    0<\ell_2<\ell_1<\ell_3\\
    0<\ell_2<\ell_3<\ell_1
\end{align*}
\end{minipage}

A couple of notes about these strings of inequalities. First, the number of strict inequalities is the same as the dimension of the corresponding cone, i.e. the columns, from left to right, correspond to rays, 2D cones, and 3D cones. Also notice that the only restriction on ordering the $\ell_i$'s is that $\ell_2\leq\ell_3$.

\end{exam}

\begin{figure}[h]
\begin{subfigure}{.48\textwidth}
    \centering
\begin{tikzpicture}
\fill (0,0) circle (0.10);
\node at (0.2,-0.3) {\cred $v_1$};
\draw[very thick] (0,0) -- (-0.5,0.5); 
\node at (-0.7,0.5) {$2$};
\draw[very thick] (0,0) -- (-0.5,-0.5); 
\node at (-0.7,-0.5) {$3$};

\draw[very thick] (0,0) -- (1.5,0);
\node at (0.75,0.3) {\cb $\ell_1$};

\fill (1.5,0) circle (0.10);
\node at (1.5,-0.3) {\cred $v_0$};
\draw[very thick] (1.5,0) -- (1.5,0.5);
\node at (1.5,0.7) {$1$};

\draw[very thick] (1.5,0) -- (3,0);
\node at (2.25,0.3) {\cb $\ell_2$};

\fill (3,0) circle (0.10);
\node at (3,-0.3) {\cred $v_2$};
\draw[very thick] (3,0) -- (3,0.5);
\node at (3,0.7) {$4$};

\draw[very thick] (3,0) -- (5,0);
\node at (4,0.3) {\cb $\ell_3-\ell_2$};

\fill (5,0) circle (0.10);
\node at (4.8,-0.3) {\cred $v_3$};
\draw[very thick] (5,0) -- (5.5,0.5); 
\node at (5.7,0.5) {$5$};
\draw[very thick] (5,0) -- (5.5,-0.5); 
\node at (5.7,-0.5) {$6$};
\end{tikzpicture}    
    \caption{Tropical curve of $\M_{0,6}^{\textrm{trop}}$ with splits $I_1=\{2,3\},~I_2=\{4,5,6\},~I_3=\{5,6\}$}
    \label{fig:CombTypeOfMzero6rad}
\end{subfigure}
\hfill
\begin{subfigure}{.48\textwidth}
    \centering
\begin{tikzpicture}[scale=0.9]
    \draw[thick] (30:3cm) -- (150:3cm);
    \node at (90:1.9cm) [] {$1$};
    \draw[thick] (150:3cm) -- (270:3cm);
    \node at (170:2.3cm) [] {$2$};
    \node at (250:2.3cm) [] {$3$};
    \draw[thick] (270:3cm) -- (30:3cm);
    \node at (290:2.3cm) [] {$4$};
    \node at (10:2.3cm) [] {$5$};
    \draw[thick] (210:1.5cm) -- (30:3cm);
    \node at (45:1.4cm) [] {$6$};
    \draw[thick] (210:1.5cm) --(330:1.5cm);
    \node at (240:1.2cm) [] {$7$};
    
    \draw[fill=black] (30:3cm) circle (3pt);
    \node at (28:3.3cm) [] {E};
    \draw[fill=black] (150:3cm) circle (3pt) node[] {};
    \node at (152:3.3cm) [] {A};
    \draw[fill=black] (270:3cm) circle (3pt) node[] {};
    \node at (270:3.4cm) [] {C};
    \draw[fill=black] (210:1.5cm) circle (3pt) node[] {};
    \node at (210:1.9cm) [] {B};
    \draw[fill=black] (330:1.5cm) circle (3pt) node[] {};
    \node at (330:1.9cm) [] {D};
    
    \node at (145:1cm) [] {I};
    \node at (340:0.7cm) [] {II};
    \node at (270:1.5cm) [] {III};
\end{tikzpicture}
    \caption{A slice of a cone of $\M_{0,6}^{\textrm{rad}}$. Rays are labeled by letters A-E, 2D cones are labeled by numbers 1-7, and 3D cones are labeled by numerals I, II, and III.}
    \label{fig:ConeOfMzero6rad}
\end{subfigure}
\caption{}
\end{figure}


\begin{exam}\label{exam:Mzero7radFubini}
Now consider the combinatorial type $\C\in\M_{0,7}^{\textrm{trop}}$ with splits $I_1=\{2,3\},~I_2=\{4,5\},~I_3=\{6,7\}$, see Figure~\ref{fig:CombTypeOfMzero7rad}. Similar to Example~\ref{exam:Mzero6rad}, in $\M_{0,7}^{\textrm{trop}}$, this combinatorial type corresponds to a single 3-dimensional cone with faces consisting of three 2-dimensional cones, and three rays. The radially aligned subdivision yields six 3-dimensional cones, twelve 2-dimensional cones, and seven rays, see Figure~\ref{fig:ConeOfMzero7rad}. If we also consider the 0-dimensional cone which is the intersection of all of these cones there are 26 in total. We may also obtain 26 by doubling the ordered Bell number on a set of three elements. The factor of 2 is due to having a distinguished least element of $\ell_0$.
\end{exam}

\begin{figure}[h]
\begin{subfigure}{.48\textwidth}
    \centering
\begin{tikzpicture}
\fill (0,0) circle (0.10);
\node at (0.2,-0.3) {\cred $v_1$};
\draw[very thick] (0,0) -- (-0.5,0.5); 
\node at (-0.7,0.5) {$2$};
\draw[very thick] (0,0) -- (-0.5,-0.5); 
\node at (-0.7,-0.5) {$3$};

\draw[very thick] (0,0) -- (1.5,0);
\node at (0.75,0.3) {\cb $\ell_1$};

\fill (1.5,0) circle (0.10);
\node at (1.8,-0.3) {\cred $v_0$};
\draw[very thick] (1.5,0) -- (1.5,0.5);
\node at (1.5,0.7) {$1$};

\draw[very thick] (1.5,0) -- (3,0);
\node at (2.25,0.3) {\cb $\ell_2$};
\draw[very thick] (1.5,0) -- (1.5,-1.5);
\node at (1.2,-0.75) {\cb $\ell_3$};

\fill (1.5,-1.5) circle (0.10);
\node at (1.8,-1.3) {\cred $v_2$};
\draw[very thick] (1.5,-1.5) -- (1,-2);
\node at (0.8,-2) {$4$};
\draw[very thick] (1.5,-1.5) -- (2,-2);
\node at (2.2,-2) {$5$};

\fill (3,0) circle (0.10);
\node at (2.8,-0.3) {\cred $v_3$};
\draw[very thick] (3,0) -- (3.5,0.5); 
\node at (3.7,0.5) {$6$};
\draw[very thick] (3,0) -- (3.5,-0.5); 
\node at (3.7,-0.5) {$7$};
\end{tikzpicture}    
    \caption{Tropical curve of $\M_{0,7}^{\textrm{trop}}$ with splits $I_1=\{2,3\},~I_2=\{4,5\},~I_3=\{6,7\}$}
    \label{fig:CombTypeOfMzero7rad}
\end{subfigure}
\hfill
\begin{subfigure}{.48\textwidth}
    \centering
\begin{tikzpicture}[scale=0.8]
    \draw[thick] (30:3cm) -- (150:3cm) -- (270:3cm) -- cycle;
    
    \foreach \x in {90,210,330}
{
    \draw[thick] (\x:1.5cm) -- (\x:-3cm);
}
    \draw[fill=black] (0:0cm) circle (3pt) node[] {};
    \draw[fill=black] (30:3cm) circle (3pt) node[] {};
    \draw[fill=black] (150:3cm) circle (3pt) node[] {};
    \draw[fill=black] (270:3cm) circle (3pt) node[] {};
    \draw[fill=black] (90:1.5cm) circle (3pt) node[] {};
    \draw[fill=black] (210:1.5cm) circle (3pt) node[] {};
    \draw[fill=black] (330:1.5cm) circle (3pt) node[] {};
\end{tikzpicture}\\
    \caption{A slice of a cone of $\M_{0,7}^{\textrm{rad}}$}
    \label{fig:ConeOfMzero7rad}
\end{subfigure}
\caption{}
\end{figure}

\begin{lem}\label{lem:MzeronradIsEqualToBKprimeAsConeComplexes}
$\Mzeronrad=\BKprime$ as polyhedral cone complexes. In particular, there is a bijection $\Psi$ between chains of flats of $K_{n-1}$ and radially aligned combinatorial types of $\Mzeronrad$.
\end{lem}
\begin{proof}
Recall that a flat $F$ of $K_{n-1}$ corresponds to a cluster graph, i.e. $F$ is a disjoint union of complete graphs.
Consider the COF of length 1 given by $F$ with vertex sets $V_1, \dots, V_k$. The abstract tropical curve corresponding to $F$ can be constructed in the following way:

Denote the root vertex $\V_0$. Attach ends labeled by $[n]\setminus\bigcup_{i=1}^{k}V_i$ and attach $k$ bounded edges of the same length where the new vertices are labeled $\V_1,\ldots,\V_k$. Finally, on each $\V_i$ attach ends marked by the vertex set $V_i$.

Now consider a COF $\F=F_1\subset\cdots\subset F_r$. Write each flat as
$$F_i=\coprod_{j=1}^{k_i} K_{I^i_j}$$
where $I^i_j$ is the vertex set of the $j$th complete graph at the $i$th step in the chain and $k_i$ is the number of connected components of $F_i$. Define for $i=1,\ldots,r$ a chain of inclusion maps
$$\phi_i:I_{l}^{i}\longmapsto I_{j}^{i+1}$$
such that $K_{I_{l}^{i}} \subset K_{I_{j}^{i+1}}$. Also define the trivial map
$$\phi_0:\emptyset\longrightarrow\coprod_{j=1}^{k_1}I_{j}^{1}.$$
The following construction gives a cone in $\Mzeronrad$ by producing a radially aligned curve belonging to it. Denote the root vertex $\V_0$ and attach ends labeled by $[n]\setminus\bigcup_{j=1}^{k_r} I_{j}^{r}$. To the root vertex attach $k_r$ bounded edges of length 1 with a vertex at the end of each edge. For each $i=r-1,\ldots,0$ and $j=1,\ldots,k_{i+1}$ attach ends given by 
$$I_{j}^{i+1}\setminus \phi_{i}^{-1}(I_{j}^{i+1}).$$ 
If $\phi_{i}^{-1}(I_{j}^{i+1})\not=\emptyset$, add a number of bounded edges of length 1 equal to the number of sets $I_{l}^i$ such that $\phi_i(I_{l}^i)\subseteq I_{j}^{i+1}$ and add a vertex at the end of each edge. After all iterations stabilize all vertices with valence 2, i.e. remove all vertices that have 2 bounded edges and no ends. The radial alignment is given by the lengths of compact edges after stabilizing the 2-valent vertices.\\

Let $\rho$ be a ray of $\Mzeronrad$, i.e. $\rho$ corresponds to an abstract tropical curve with $d$ bounded edges of the same length and disjoint splits $I_1,\ldots,I_d$. The corresponding flat is $F_{\rho}=\coprod_{j=1}^d K_{I_j}$. Note that rays of $\Mzerontrop$ correspond to abstract tropical curves with exactly one bounded edge. Later we see that the corresponding flats are precisely the $1$-connected flats of $K_{n-1}$.

Let $\C$ be an arbitrary radially aligned combinatorial type with and ordered partition of the vertices given by the distances $d_1,\ldots,d_r$. Define the $i$th \emph{level} $\mathscr{L}_i$ in the following way:
\begin{enumerate}
    \item Delete all vertices and ends strictly within the radius $d_i$ of the root vertex and delete all bounded edges which are at least partially contained within radius $d_i$.
    \item For each remaining connected component record the set of ends.
    \item $\mathscr{L}_i$ is the set containing the sets from the previous step.
\end{enumerate}
Then the COF $\F_\C$ is 
$$ \coprod_{I\in \mathscr{L}_{r}} K_{I} \subset \coprod_{I\in \mathscr{L}_{r-1}} K_{I} \subset \cdots \subset \coprod_{I\in \mathscr{L}_{1}} K_{I}.$$
\end{proof}


\begin{exam}
Consider the radially aligned tropical curve as pictured in $\M_{0,8}^{\textrm{rad}}$ in Figure~\ref{fig:CombTypeOfMzero8rad}. The ordered partition of vertices is $L_1=\{v_1\},L_2=\{v_2,v_3\}$ and $L_3=\{v_4,v_5\}$ with respective distances $d_1=1,d_2=2,$ and $d_3=3$. Figure~\ref{fig:LevelDepictions} depicts the construction of $\mathscr{L}_1=\{ \{4,5,6,7\} , \{2,3,8\} \}$, $\mathscr{L}_2=\{ \{4,5,6,7\} , \{2,3\} \}$ and $\mathscr{L}_3=\{ \{4,5\}, \{6,7\} \}$. This gives us the COF of length 3, $\F$ 
$$ K_{\{4,5\}} \sqcup K_{\{6,7\}}  \subset K_{\{4,5,6,7\}} \sqcup K_{\{2,3\}} \subset K_{\{4,5,6,7\}} \sqcup K_{\{2,3,8\}}.$$

\begin{figure}[h]
    \centering
\begin{tikzpicture}
\fill (2,0) circle (0.10);
\node at (2,-0.3) {\cred $v_0$};
\draw[very thick] (2,0) -- (2,0.5);
\node at (2,0.7) {$1$};

\fill (3,0) circle (0.10);
\node at (3,-0.3) {\cred $v_1$};
\draw[very thick] (3,0) -- (3,0.5);
\node at (3,0.7) {$8$};

\fill (4,0) circle (0.10);
\node at (3.8,-0.3) {\cred $v_2$};
\draw[very thick] (4,0) -- (4.5,0.5);
\node at (4.7,0.5) {$2$};
\draw[very thick] (4,0) -- (4.5,-0.5);
\node at (4.7,-0.5) {$3$};

\fill (0,0) circle (0.10);
\node at (-0.4,0) {\cred $v_3$};

\fill (0,-1) circle (0.10);
\node at (-0.3,-0.8) {\cred $v_4$};
\draw[very thick] (0,-1) -- (-0.5,-1.5);
\node at (-0.7,-1.5) {$4$};
\draw[very thick] (0,-1) -- (0.5,-1.5);
\node at (0.7,-1.5) {$5$};

\fill (0,1) circle (0.10);
\node at (-0.3,0.8) {\cred $v_5$};
\draw[very thick] (0,1) -- (-0.5,1.5);
\node at (-0.7,1.5) {$6$};
\draw[very thick] (0,1) -- (0.5,1.5);
\node at (0.7,1.5) {$7$};

\draw[very thick] (0,0) -- (0,1);
\node at (0.2,0.5) {\cb $1$};
\draw[very thick] (0,0) -- (0,-1);
\node at (0.2,-0.5) {\cb $1$};
\draw[very thick] (0,0) -- (2,0);
\node at (1,0.3) {\cb $2$};
\draw[very thick] (2,0) -- (3,0);
\node at (2.5,0.3) {\cb $1$};
\draw[very thick] (3,0) -- (4,0);
\node at (3.5,0.3) {\cb $1$};
\end{tikzpicture}\\
    \caption{Tropical curve in $\M_{0,8}^{\textrm{rad}}$}
    \label{fig:CombTypeOfMzero8rad}
\end{figure}
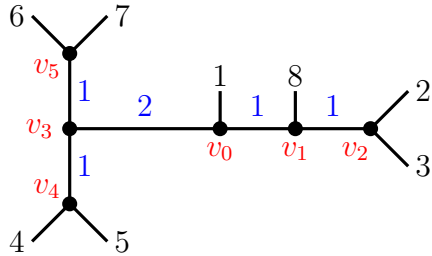

\begin{figure}[h]
\begin{subfigure}{.32\textwidth}
    \centering
\begin{tikzpicture}[scale=0.9]
\fill[color = red] (2,0) circle (0.10);
\draw[very thick, color = red] (2,0) -- (2,0.5);
\node[color = red] at (2,0.7) {$1$};

\fill (3,0) circle (0.10);
\draw[very thick] (3,0) -- (3,0.5);
\node at (3,0.7) {$8$};

\fill (4,0) circle (0.10);
\draw[very thick] (4,0) -- (4.5,0.5);
\node at (4.7,0.5) {$2$};
\draw[very thick] (4,0) -- (4.5,-0.5);
\node at (4.7,-0.5) {$3$};

\fill (0,0) circle (0.10);

\fill (0,-1) circle (0.10);
\draw[very thick] (0,-1) -- (-0.5,-1.5);
\node at (-0.7,-1.5) {$4$};
\draw[very thick] (0,-1) -- (0.5,-1.5);
\node at (0.7,-1.5) {$5$};

\fill (0,1) circle (0.10);
\draw[very thick] (0,1) -- (-0.5,1.5);
\node at (-0.7,1.5) {$6$};
\draw[very thick] (0,1) -- (0.5,1.5);
\node at (0.7,1.5) {$7$};

\draw[very thick] (0,0) -- (0,1);
\draw[very thick] (0,0) -- (0,-1);
\draw[very thick, color = red] (0.1,0) -- (2,0);
\draw[very thick, color = red] (2,0) -- (2.9,0);
\draw[very thick] (3,0) -- (4,0);
\end{tikzpicture}
    \caption{Level 1 designated by the dual graph in black.}
    \label{fig:Level1}
\end{subfigure}\hfill
\begin{subfigure}{.32\textwidth}
    \centering
\begin{tikzpicture}[scale=0.9]
\fill[color = red] (2,0) circle (0.10);
\draw[very thick, color = red] (2,0) -- (2,0.5);
\node[color = red] at (2,0.7) {$1$};

\fill[color = red] (3,0) circle (0.10);
\draw[very thick, color = red] (3,0) -- (3,0.5);
\node[color = red] at (3,0.7) {$8$};

\fill (4,0) circle (0.10);
\draw[very thick] (4,0) -- (4.5,0.5);
\node at (4.7,0.5) {$2$};
\draw[very thick] (4,0) -- (4.5,-0.5);
\node at (4.7,-0.5) {$3$};

\fill (0,0) circle (0.10);

\fill (0,-1) circle (0.10);
\draw[very thick] (0,-1) -- (-0.5,-1.5);
\node at (-0.7,-1.5) {$4$};
\draw[very thick] (0,-1) -- (0.5,-1.5);
\node at (0.7,-1.5) {$5$};

\fill (0,1) circle (0.10);
\draw[very thick] (0,1) -- (-0.5,1.5);
\node at (-0.7,1.5) {$6$};
\draw[very thick] (0,1) -- (0.5,1.5);
\node at (0.7,1.5) {$7$};

\draw[very thick] (0,0) -- (0,1);
\draw[very thick] (0,0) -- (0,-1);
\draw[very thick, color = red] (0.1,0) -- (2,0);
\draw[very thick, color = red] (2,0) -- (3,0);
\draw[very thick, color = red] (3,0) -- (3.9,0);
\end{tikzpicture}
    \caption{Level 2 designated by the dual graph in black.}
    \label{fig:Level2}
\end{subfigure}\hfill
\begin{subfigure}{.32\textwidth}
    \centering
\begin{tikzpicture}[scale=0.9]
\fill[color = red] (2,0) circle (0.10);
\draw[very thick, color = red] (2,0) -- (2,0.5);
\node[color = red] at (2,0.7) {$1$};

\fill[color = red] (3,0) circle (0.10);
\draw[very thick, color = red] (3,0) -- (3,0.5);
\node[color = red] at (3,0.7) {$8$};

\fill[color = red] (4,0) circle (0.10);
\draw[very thick, color = red] (4,0) -- (4.5,0.5);
\node[color = red] at (4.7,0.5) {$2$};
\draw[very thick, color = red] (4,0) -- (4.5,-0.5);
\node[color = red] at (4.7,-0.5) {$3$};

\fill[color = red] (0,0) circle (0.10);

\fill (0,-1) circle (0.10);
\draw[very thick] (0,-1) -- (-0.5,-1.5);
\node at (-0.7,-1.5) {$4$};
\draw[very thick] (0,-1) -- (0.5,-1.5);
\node at (0.7,-1.5) {$5$};

\fill (0,1) circle (0.10);
\draw[very thick] (0,1) -- (-0.5,1.5);
\node at (-0.7,1.5) {$6$};
\draw[very thick] (0,1) -- (0.5,1.5);
\node at (0.7,1.5) {$7$};

\draw[very thick, color = red] (0,0) -- (0,0.9);
\draw[very thick, color = red] (0,0) -- (0,-0.9);
\draw[very thick, color = red] (0,0) -- (2,0);
\draw[very thick, color = red] (2,0) -- (3,0);
\draw[very thick, color = red] (3,0) -- (4,0);
\end{tikzpicture}
    \caption{Level 3 designated by the dual graph in black.}
    \label{fig:Level3}
\end{subfigure}
\caption{}
\label{fig:LevelDepictions}
\end{figure}

Now let's begin with the above COF of length $r=3$, $\F$. The index sets are listed below. Note that $k_1=k_2=k_3=2$.

$$
\begin{array}{lllll}
    I_{1}^{1} = \{4,5\}  &&I_{1}^{2} = \{4,5,6,7\}  &&I_{1}^{3} = \{4,5,6,7\} \\
    I_{2}^{1} = \{6,7\}  &&I_{2}^{2} = \{2,3\}  &&I_{2}^{3} = \{2,3,8\}
\end{array}
$$

In this example, the construction will take 5 steps where the final step is stabilizing the 2-valent vertices. See Figure~\ref{fig:ConstructionAllSteps} for a depiction of this construction. The first step is when we create the root vertex and attach the end $1$ which is given by $[8]\setminus \left(I_{3}^{1}\cup I_{3}^{2}\right)$. We also attach $k_3=2$ bounded edges of length 1.

In the second step we attach the end $8$, given by $I_{2}^{3}\setminus\phi_{2}^{-1}\left(I_{2}^{3}\right)$, to a vertex. Then since both $\phi_{2}^{-1}\left(I_{1}^{3}\right)$ and $\phi_{2}^{-1}\left(I_{2}^{3}\right)$ are nonempty we add a single bounded edge to both of the vertices from the first step.

For the third step we will start with the branch that has the end $8$. To the new vertex we attach the ends $2,3$ and no bounded edges. On the other branch we attach no ends and two bounded edges. There are two bounded edges because $\phi_1(I_{1}^1)\subseteq I_{1}^{2}$ and $\phi_1(I_{2}^1)\subseteq I_{1}^{2}$. 

The fourth step has us attaching the final ends. We attach the ends $4,5$ to one of the new vertices from step 3 and attach $6,7$ to the other vertex. Finally we omit the only 2-valent vertex.

\begin{figure}[h]
    \centering
\begin{tikzpicture}[]
\begin{scope}[scale=1]

\fill[color = blue] (2,0) circle (0.10);
\node at (2,-0.3) {\cred $v_0$};
\draw[very thick, color = blue] (2,0) -- (2,0.5);
\node[color = blue] at (2,0.7) {$1$};

\fill[color = blue] (3,0) circle (0.10);

\fill[color = blue] (1,0) circle (0.10);

\draw[very thick, color = blue] (1,0) -- (2,0);
\draw[very thick, color = blue] (2,0) -- (3,0);


\begin{scope}[shift={(5,0)}]
\fill[] (2,0) circle (0.10);
\draw[very thick,] (2,0) -- (2,0.5);
\node[] at (2,0.7) {$1$};

\fill[] (3,0) circle (0.10);
\draw[very thick, color = blue] (3,0.1) -- (3,0.5);
\node[color = blue] at (3,0.7) {$8$};

\fill[] (1,0) circle (0.10);

\fill[color = blue] (4,0) circle (0.10);

\fill[color = blue] (0,0) circle (0.10);

\draw[very thick, color = blue] (0,0) -- (0.9,0);
\draw[very thick] (1,0) -- (2,0);
\draw[very thick] (2,0) -- (3,0);
\draw[very thick, color = blue] (3.1,0) -- (3.9,0);
\end{scope}


\begin{scope}[shift={(11,0)}]
\fill[] (2,0) circle (0.10);
\draw[very thick,] (2,0) -- (2,0.5);
\node[] at (2,0.7) {$1$};

\fill[] (3,0) circle (0.10);
\draw[very thick] (3,0) -- (3,0.5);
\node[] at (3,0.7) {$8$};

\fill[] (1,0) circle (0.10);

\fill[] (4,0) circle (0.10);
\draw[very thick, color = blue] (4.071,0.071) -- (4.5,0.5);
\node[color = blue] at (4.7,0.5) {$2$};
\draw[very thick, color = blue] (4.071,-0.071) -- (4.5,-0.5);
\node[color = blue] at (4.7,-0.5) {$3$};

\fill[] (0,0) circle (0.10);

\fill[color = blue] (0,-1) circle (0.10);

\fill[color = blue] (0,1) circle (0.10);

\draw[very thick, color = blue] (0,0.1) -- (0,0.9);
\draw[very thick, color = blue] (0,-0.1) -- (0,-0.9);
\draw[very thick] (0,0) -- (1,0);
\draw[very thick] (1,0) -- (2,0);
\draw[very thick] (2,0) -- (3,0);
\draw[very thick] (3,0) -- (4,0);
\end{scope}


\begin{scope}[shift={(2,-4)}]
\fill[] (2,0) circle (0.10);
\draw[very thick,] (2,0) -- (2,0.5);
\node[] at (2,0.7) {$1$};

\fill[] (3,0) circle (0.10);
\draw[very thick] (3,0) -- (3,0.5);
\node[] at (3,0.7) {$8$};

\fill[] (1,0) circle (0.10);

\fill[] (4,0) circle (0.10);
\draw[very thick] (4,0) -- (4.5,0.5);
\node[] at (4.7,0.5) {$2$};
\draw[very thick] (4,0) -- (4.5,-0.5);
\node[] at (4.7,-0.5) {$3$};

\fill[] (0,0) circle (0.10);

\fill[] (0,-1) circle (0.10);
\draw[very thick, color = blue] (-0.071,-1.071) -- (-0.5,-1.5);
\node[color = blue] at (-0.7,-1.5) {$4$};
\draw[very thick, color = blue] (0.071,-1.071) -- (0.5,-1.5);
\node[color = blue] at (0.7,-1.5) {$5$};

\fill[] (0,1) circle (0.10);
\draw[very thick, color = blue] (-0.071,1.071) -- (-0.5,1.5);
\node[color = blue] at (-0.7,1.5) {$6$};
\draw[very thick, color = blue] (0.071,1.071) -- (0.5,1.5);
\node[color = blue] at (0.7,1.5) {$7$};

\draw[very thick] (0,0) -- (0,1);
\draw[very thick] (0,0) -- (0,-1);
\draw[very thick] (0,0) -- (1,0);
\draw[very thick] (1,0) -- (2,0);
\draw[very thick] (2,0) -- (3,0);
\draw[very thick] (3,0) -- (4,0);
\end{scope}


\begin{scope}[shift = {(10,-4)}]
\fill (2,0) circle (0.10);
\draw[very thick] (2,0) -- (2,0.5);
\node at (2,0.7) {$1$};

\fill (3,0) circle (0.10);
\draw[very thick] (3,0) -- (3,0.5);
\node at (3,0.7) {$8$};

\fill (4,0) circle (0.10);
\draw[very thick] (4,0) -- (4.5,0.5);
\node at (4.7,0.5) {$2$};
\draw[very thick] (4,0) -- (4.5,-0.5);
\node at (4.7,-0.5) {$3$};

\fill (0,0) circle (0.10);

\fill (0,-1) circle (0.10);
\draw[very thick] (0,-1) -- (-0.5,-1.5);
\node at (-0.7,-1.5) {$4$};
\draw[very thick] (0,-1) -- (0.5,-1.5);
\node at (0.7,-1.5) {$5$};

\fill (0,1) circle (0.10);
\draw[very thick] (0,1) -- (-0.5,1.5);
\node at (-0.7,1.5) {$6$};
\draw[very thick] (0,1) -- (0.5,1.5);
\node at (0.7,1.5) {$7$};

\draw[very thick] (0,0) -- (0,1);
\node at (0.2,0.5) {\cb $1$};
\draw[very thick] (0,0) -- (0,-1);
\node at (0.2,-0.5) {\cb $1$};
\draw[very thick] (0,0) -- (2,0);
\node at (1,0.3) {\cb $2$};
\draw[very thick] (2,0) -- (3,0);
\node at (2.5,0.3) {\cb $1$};
\draw[very thick] (3,0) -- (4,0);
\node at (3.5,0.3) {\cb $1$};
\end{scope}

\end{scope}
\end{tikzpicture}\\
    \caption{Construction of a radially aligned tropical curve from a chain-of-flats}
    \label{fig:ConstructionAllSteps}
\end{figure}
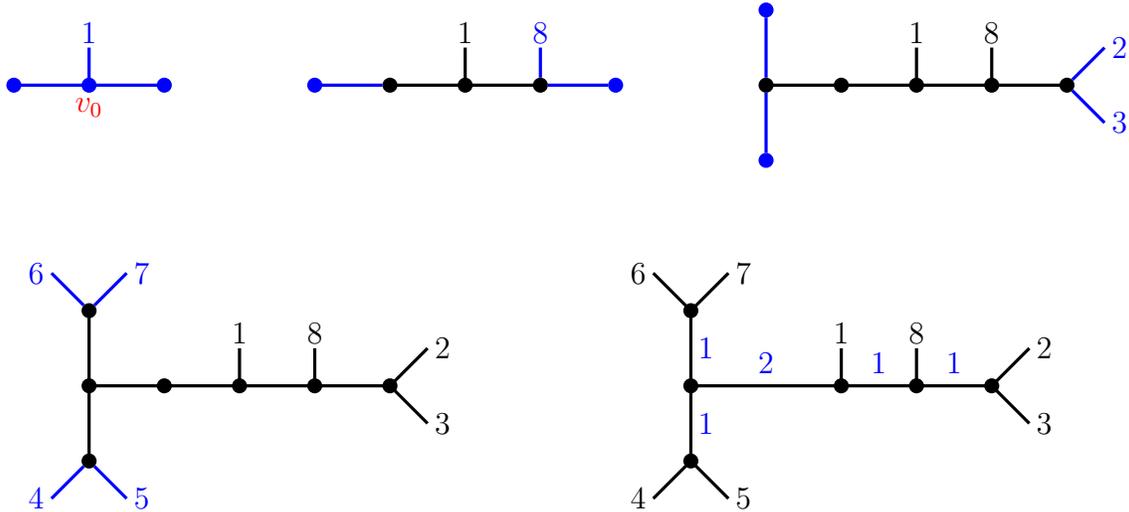

\end{exam}
 
\begin{rem}\label{rem:MzeronradIsEqualToBKprimeAsFans}
Using the bijection described in the proof of Lemma~\ref{lem:MzeronradIsEqualToBKprimeAsConeComplexes} we get an isomorphism of vector spaces, denoted using the same letter,
\begin{center}
    \begin{tikzcd}[row sep=1cm, column sep=1cm]
\Mzeronrad \arrow[d, hook] \arrow[r,"\Psi"', "="] \drar[phantom, "\square"]& \BKprime \arrow[d, hook] \\
Q_n \arrow[r,"\Psi"', "\cong"] & \RR^{|E(K_{n-1})|}/L 
\end{tikzcd}
\end{center}
which respects the cone complex structures of $\Mzeronrad$ and $\BKprime$. Therefore we may write $\Mzeronrad=\BKprime$ as polyhedral fans.
\end{rem}

\begin{exam}\label{exam:Mzero5tropEqualsBK4prime}
It is known that the cone complex of $\M_{0,5}^{\textrm{trop}}$ is given by the cone over the Petersen graph. We see in the Bergman fan $B'(K_4)$ that 3 edges are subdivided, see Figure~\ref{fig:BergmanFanOfK4}. Label the lattice of flats of $K_4$ in the following way:
\begin{align*}
    &\textrm{Rank } 1:~ F_1=\{e_{23}\},~F_2=\{e_{24}\},~F_3=\{e_{34}\},~F_4=\{e_{35}\},~F_5=\{e_{45}\},~F_6=\{e_{25}\}\\
    &\textrm{Rank } 2 \textrm{ connected}:~ F_7=\{e_{23},e_{24},e_{34}\},~ F_8=\{e_{23},e_{35},e_{25}\},~ F_9=\{e_{34},e_{35},e_{45}\},\\
    & \hspace{105pt}F_{10}=\{e_{24},e_{45},e_{25}\}\\
    &\textrm{Rank } 2 \textrm{ disconnected}:~ F_{11}=\{e_{23},e_{45}\},~ F_{12}=\{e_{24},e_{35}\},~ F_{13}=\{e_{25},e_{34}\}.
\end{align*}
See Figure~\ref{fig:FlatsOfK4} for a visual representation of the flats $F_1$, $F_7$, and $F_{11}$. Consider the top-dimensional cone $\sigma$ in $\M_{0,5}^{\textrm{trop}}$ with the combinatorial type that has a root vertex $\V_0$ with two bounded edges and adjacent vertices $\V_1$ and $\V_2$ with ends marked by $I_1=\{2,3\}$ and $I_2=\{4,5\}$. An abstract tropical curve $\C$ with this combinatorial type has edge lengths $\ell_{1},\ell_{2}\in\RR^+$, see Figure~\ref{fig:CombTypeOfMzero5trop}.

In $B'(K_4)$, and therefore $\M_{0,5}^{\textrm{rad}}$, we see that this cone is subdivided into $\sigma_1=\textrm{cone}(\rho_{F_{1}},\rho_{F_{11}})$ and $\sigma_2=\textrm{cone}(\rho_{F_{11}},\rho_{F_{5}})$ with their intersection being a ray $\rho=\rho_{F_{11}}$. It turns out that $\rho$ corresponds to $\C$ where $\ell_{1}=\ell_{2}$ and $\sigma_i$ is the cone corresponding to the abstract tropical curve $\C$ where $\ell_{i}>\ell_{j}$. \\

\begin{figure}[h]
\begin{subfigure}{.48\textwidth}
    \centering
\begin{tikzpicture}[scale=0.75]
    \draw[thick] (18:4cm) -- (90:4cm) -- (162:4cm) -- (234:4cm) --
(306:4cm) -- cycle;
    \draw[thick] (18:2cm) -- (162:2cm) -- (306:2cm) -- (90:2cm) --
(234:2cm) -- cycle;

    \draw[fill=black] (90:3cm) circle (3pt) node[label=right:$F_{11}$] {};
    \draw[fill=black] (90:0.618cm) circle (3pt) node[label=below:$F_{13}$] {};
    \draw[fill=black] (270:3.236cm) circle (3pt) node[label=below:$F_{12}$] {};

    \foreach \x in {18,90,162,234,306}
{
    \draw[thick] (\x:2cm) -- (\x:4cm);
}
    \draw[fill=black] (18:2cm) circle (3pt) node[label=90:$F_{6}$] {};
    \draw[fill=black] (90:2cm) circle (3pt) node[label=180:$F_{5}$] {};
    \draw[fill=black] (162:2cm) circle (3pt) node[label=270:$F_{3}$] {};
    \draw[fill=black] (234:2cm) circle (3pt) node[label=360:$F_{10}$] {};
    \draw[fill=black] (306:2cm) circle (3pt) node[label=30:$F_{9}$] {};
    \draw[fill=black] (18:4cm) circle (3pt) node[label=18:$F_{8}$] {};
    \draw[fill=black] (90:4cm) circle (3pt) node[label=90:$F_{1}$] {};
    \draw[fill=black] (162:4cm) circle (3pt) node[label=162:$F_{7}$] {};
    \draw[fill=black] (234:4cm) circle (3pt) node[label=234:$F_{2}$] {};
    \draw[fill=black] (306:4cm) circle (3pt) node[label=306:$F_{4}$] {};
\end{tikzpicture}
    \caption{A slice of the Bergman fan $\B'(K_4)$ with rays labeled by their corresponding flats.}
    \label{fig:BergmanFanOfK4}
\end{subfigure}
\hfill
\begin{subfigure}{.48\textwidth}
    \centering
\begin{tikzpicture}[scale=1.3]
\fill (0,0) circle (0.10);
\node at (0.2,-0.3) {\cred $v_1$};
\draw[very thick] (0,0) -- (-0.5,0.5); 
\node at (-0.7,0.5) {$2$};
\draw[very thick] (0,0) -- (-0.5,-0.5); 
\node at (-0.7,-0.5) {$3$};

\draw[very thick] (0,0) -- (1.5,0);
\node at (0.75,0.3) {\cb $\ell_1$};

\fill (1.5,0) circle (0.10);
\node at (1.5,-0.3) {\cred $v_0$};
\draw[very thick] (1.5,0) -- (1.5,0.5);
\node at (1.5,0.7) {$1$};

\draw[very thick] (1.5,0) -- (3,0);
\node at (2.25,0.3) {\cb $\ell_2$};

\fill (3,0) circle (0.10);
\node at (2.8,-0.3) {\cred $v_2$};
\draw[very thick] (3,0) -- (3.5,0.5); 
\node at (3.7,0.5) {$4$};
\draw[very thick] (3,0) -- (3.5,-0.5); 
\node at (3.7,-0.5) {$5$};
\end{tikzpicture}
    \caption{A tropical curve of $\M_{0,5}^{\textrm{trop}}$ with splits $I_1=\{2,3\}$ and $I_2=\{4,5\}$.}
    \label{fig:CombTypeOfMzero5trop}
\end{subfigure}
\caption{}
\end{figure}

\begin{figure}[h]
\begin{subfigure}{.32\textwidth}
    \centering
\begin{tikzpicture}[scale=1.8]
    \node (v2) at (1,1.732) [circle, draw = black, inner sep =  2pt, outer sep = 0.5pt, minimum size = 4mm, line width = 1pt] {$2$};
    \node (v3) at (2,0) [circle, draw = black, inner sep =  2pt, outer sep = 0.5pt, minimum size = 4mm, line width = 1pt] {$3$};
    \node (v4) at (0,0) [circle, draw = black, inner sep =  2pt, outer sep = 0.5pt, minimum size = 4mm, line width = 1pt] {$4$};
    \node (v5) at (1,0.577) [circle, draw = black, inner sep =  2pt, outer sep = 0.5pt, minimum size = 4mm, line width = 1pt] {$5$};
    
    \draw[line width = 1pt] (v2) -- node [midway, right] {$e_{23}$} (v3);
\end{tikzpicture}
    \caption{Flat $F_1$}
    \label{fig:FlatF1ofK4}
\end{subfigure}\hfill
\begin{subfigure}{.32\textwidth}
    \centering
\begin{tikzpicture}[scale=1.8]
    \node (v2) at (1,1.732) [circle, draw = black, inner sep =  2pt, outer sep = 0.5pt, minimum size = 4mm, line width = 1pt] {$2$};
    \node (v3) at (2,0) [circle, draw = black, inner sep =  2pt, outer sep = 0.5pt, minimum size = 4mm, line width = 1pt] {$3$};
    \node (v4) at (0,0) [circle, draw = black, inner sep =  2pt, outer sep = 0.5pt, minimum size = 4mm, line width = 1pt] {$4$};
    \node (v5) at (1,0.577) [circle, draw = black, inner sep =  2pt, outer sep = 0.5pt, minimum size = 4mm, line width = 1pt] {$5$};
    
    \draw[line width = 1pt] (v2) -- node [midway, right] {$e_{23}$} (v3);
    \draw[line width = 1pt] (v2) -- node [midway, left] {$e_{24}$} (v4);
    \draw[line width = 1pt] (v3) -- node [midway, below] {$e_{34}$} (v4);
\end{tikzpicture}
    \caption{Flat $F_7$}
    \label{fig:FlatF7ofK4}
\end{subfigure}\hfill
\begin{subfigure}{.32\textwidth}
    \centering
\begin{tikzpicture}[scale=1.8]
    \node (v2) at (1,1.732) [circle, draw = black, inner sep =  2pt, outer sep = 0.5pt, minimum size = 4mm, line width = 1pt] {$2$};
    \node (v3) at (2,0) [circle, draw = black, inner sep =  2pt, outer sep = 0.5pt, minimum size = 4mm, line width = 1pt] {$3$};
    \node (v4) at (0,0) [circle, draw = black, inner sep =  2pt, outer sep = 0.5pt, minimum size = 4mm, line width = 1pt] {$4$};
    \node (v5) at (1,0.577) [circle, draw = black, inner sep =  2pt, outer sep = 0.5pt, minimum size = 4mm, line width = 1pt] {$5$};
    
    \draw[line width = 1pt] (v2) -- node [midway, right] {$e_{23}$} (v3);
    \draw[line width = 1pt] (v4) -- node [midway, below right] {$e_{45}$} (v5);
\end{tikzpicture}    
    \caption{Flat $F_{11}$}
    \label{fig:FlatF11ofK4}
\end{subfigure}
\caption{Some flats of $K_4$}
\label{fig:FlatsOfK4}
\end{figure}
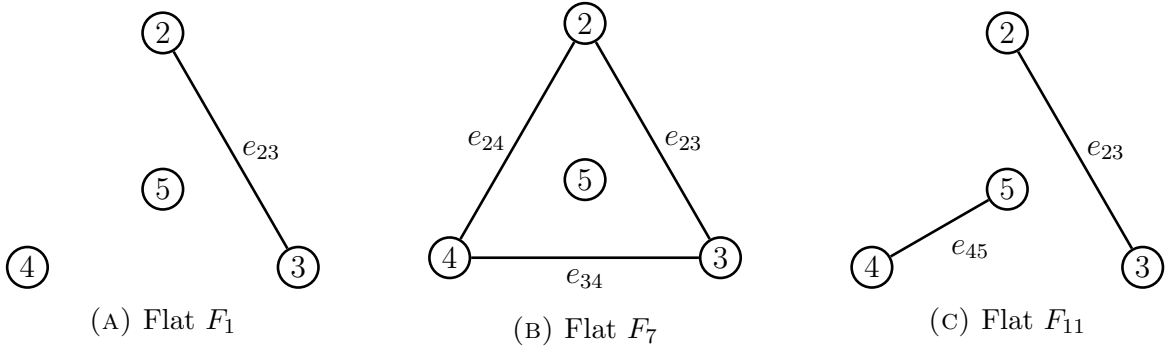



\end{exam}

\subsection{Moduli spaces of rational graphically stable tropical curves as Bergman fans}

\subsubsection{The image of $\Mzeronrad$ equals $\BGammaprime$}

As a generalization of \cite{cavalieri2016moduli} we define the space of graphically stable tropical curves and investigate its ability to be embedded as a balanced fan. In particular, we explore the relationship between $\MzeroGtrop$, $\MzeroGrad$, and $\BGammaprime$. Let $\Gamma$ be a simple connected graph whose nodes are in bijection with ends $2,\ldots,n$ of $\C$. 


\begin{defn}\label{def:GammaStability}
Let $\V$ be a non-root vertex of $\C\in\Mzerontrop$ with $d$ bounded edges and ends given by $I\subset [n].$
We say that $\C$ is $\Gamma$-\emph{stable at} $\V$ if
\begin{itemize}
    \item $d>2$;
    \item $d=2$, then $|I|\neq0$; or
    \item $d=1$, then there exists $e_{ij}\in E(\Gamma)$ for $i,j\in I$.
\end{itemize}
The root vertex, $\V_0$, is $\Gamma$-stable if it has at least 2 bounded edges or at least 1 end attached to $\V_0$. We say that $\C$ is $\Gamma$-\emph{stable} if $\C$ is $\Gamma$-stable at all $\V$.

We define $\MzeroGtrop$ to be the parameter space of all rational $n$-marked $\Gamma$-stable abstract tropical curves. Similarly, we define $\MzeroGrad$ to be the parameter space of rational $n$-marked $\Gamma$-stable radially aligned abstract tropical curves.
\end{defn}


As before, both of these spaces are well-defined as cone complexes but not necessarily as balanced fans. 

\begin{defn}
We define the \emph{reduction morphism}
\begin{align}\label{eqn:ReductionMorphism}
    c_\Gamma : \Mzerontrop\longrightarrow\MzeroGtrop
\end{align}
which successively contracts bounded edges adjacent to $\Gamma$-unstable vertices.
\end{defn}

\begin{exam}\label{exam:GraphicAndHassett}
Let $\Gamma$ be a path of length 2. Then $\MzeroGtrop$ is exactly the tropical moduli space of weighted stable tropical curves $\M_{0,\A}^{\textrm{trop}}$ with weight data $\A=(1,1,1/2,1/2)$. The fan associated to this moduli space sits in $\RR$ and contains a node at the origin and two rays pointing in opposite directions.
\end{exam}

The previous example showed that the set of moduli spaces of weighted stable tropical curves and the set of moduli spaces of $\Gamma$ stable tropical curves have an intersection. The next two examples show that neither is contained in the other.

\begin{exam}\label{exam:GraphicNotHassett}
Let $\Gamma$ be a complete bipartite graph on 4 vertices. Consider the labeling as in Figure~\ref{fig:LabeledK4} with the edges $e_{25}$ and $e_{34}$ removed. In this case, $\MzeroGtrop$ is not isomorphic to a tropical moduli space with weighted points.
\end{exam}

\begin{exam}\label{exam:HassettNotGraphic}
Consider the weight data $\A=(1,1,1/2,1/2,1/2)$. A likely choice of a corresponding graph would be the graph $\Gamma$ obtained by removing the set of edges $\{e_{34},e_{35},e_{45}\}$ from $K_4$. However, that graph corresponds to the weight data $\A'=(1,1,\varepsilon,\varepsilon,\varepsilon)$. We can see the difference by looking at the combinatorial type with split $I=\{3,4,5\}$. It is $\A$-stable but not $\Gamma$-stable nor $\A'$-stable.
\end{exam}

To relate the theory of Bergman fans to these new graphically stable moduli spaces we need to understand what $\Gamma$-stability means in terms of COFs of $K_{n-1}$.

\begin{defn}\label{def:FlatualGammaStability}
A flat $F$ of $K_{n-1}$ is \emph{$\Gamma$-stable} if the combinatorial type $\C_F$ is $\Gamma$-stable.
\end{defn}

Recall that a flat of $\Gamma$ can be thought of as a flat of $K_{n-1}$ restricted to the edge set of $\Gamma$, i.e. a flat of $\Gamma$ is $F\cap\Gamma$ where $F$ is a flat of $K_{n-1}$. Consider a ray $\rho$ with splits $I_1,\ldots,I_d$ and it's corresponding cluster graph $F_\rho=\coprod_{j=1}^d K_{I_{j}}$. Then $\rho$ is $\Gamma$-unstable if and only if there exists $I_j$ such that $K_{I_{j}}\cap\Gamma$ has no edges.

\begin{rem}\label{rem:GammaStabilityDeletionOfCliques}
In this matroidal notion, $\Gamma$-stability can be thought of as deletion of cliques of $K_{n-1}$.
\end{rem}

Now consider the map
\begin{align}\label{eqn:projgOnBergmanFanVectorSpaces}
    \projg:\RR^{|E(K_{n-1})|}/L\longrightarrow\RR^{|E(K_{n-1})|}/L/S
\end{align}
where $L$ is the lineality space spanned by the vector $(1,1,\ldots,1)$ and $S=\textrm{span}\{v_e|e\not\in\Gamma\}$ is the span of basis vectors corresponding to edges not in $\Gamma$. Note that $\projg$ is the natural projection map that forgets the coordinates corresponding to edges that are not in $\Gamma$. \\ 

Simultaneously, we may define 
\begin{align}\label{eqn:projgtildeOnModuliVectorSpaces}
    \projgtilde:Q_n\longrightarrow Q_n/U 
\end{align}

where $U$ is the linear span of $\Gamma$-unstable rays of $\Mzerontrop$. As described in Remark~\ref{rem:BasisOfRaysEqualsDivisorCombTypes}, we associate to a basis of $Q_n$ a set of combinatorial types of curves with splits $I$ of size 2. A split of size 2 corresponds to an edge of $\Gamma$. Thus $U$ is generated by combinatorial types corresponding to the edges removed from $\Gamma$.

\begin{prop}\label{prop:projgOfMzeronradIsEqualToprojgOfBKprimeAsFans}
The fans $\projg(\BKprime)$ and $\projgtilde(\Mzeronrad)$ and $\projgtilde(\Mzerontrop)$ have the same support. Furthermore $\projg(\BKprime)=\projgtilde(\Mzeronrad)$ as fans.
\end{prop}
\begin{proof}
It is clear from the discussion leading up to this proposition that the diagram in Figure~\ref{fig:PropositionDiagram} is commutative and since $\Psi$ respects the fan structures upstairs, we obtain an isomorphism downstairs that also respects the fan structures.
\end{proof}

\begin{figure}[h]
    \centering
\begin{tikzcd}[row sep=1cm, column sep=1cm]
Q_n \arrow{d}{\projgtilde} \arrow[r,"\Psi"', "\cong"] \drar[phantom, "\square"] & \RR^{|E(K_{n-1})|}/L \arrow{d}{\projg} \\
Q_n/U \arrow{r}{\cong} & \RR^{|E(K_{n-1})|}/L/S 
\end{tikzcd}
    \caption{Diagram for Proposition~\ref{prop:projgOfMzeronradIsEqualToprojgOfBKprimeAsFans}}
    \label{fig:PropositionDiagram}
\end{figure}

In general, the cone complex structures of $\projg(\BKprime)$, $\BGammaprime$, and $\MzeroGrad$ do not all coincide. 
First we investigate $\projg$ and the relationship between $\projg(\BKprime)$ and $\BGammaprime$. From this point on we will use $\projg$ to refer to both projection maps above.\\

\begin{lem}\label{lem:ProjectionDoesNotContractAllTopDimCones}
The map $\projg$ doesn't contract all top-dimensional cones of $\Mzeronrad$.
\end{lem}
\begin{proof}
Since $\BKprime$ has the same as the image of $\Mzeronrad$ we need only construct a COF $\F$ of $K_{n-1}$ of length $n-3$ that remains length $n-3$ when restricting the edge set to $\Gamma$. Then $\F$ corresponds to a radially aligned combinatorial type $\C_\F$ which is a trivalent tree such that when you shrink all but one bounded edge the resulting combinatorial type is $\Gamma$-stable.\\

\noindent \emph{Construction:}\\
Note that $\Gamma$ is connected and let $T$ be a spanning tree of $\Gamma$.
Fix adjacent vertices $i$ and $j$ in $T$. Denote $I_1=\{i,j\}$ and define vertex sets recursively $I_k=I_{k-1}\bigcup\{i_{k-1}\}$ where $\{i_1,\ldots,i_{n-4}\}=[n]\setminus\{1,i,j\}$ such that $T_{I_k}=T\cap K_{I_k}$ is connected. Then $\F$ is our desired COF of $K_{n-1}$
\begin{align}
    \F:& ~ K_{I_1}\subset \cdots \subset K_{I_{n-3}}.
\end{align}
This can be seen by the following argument:

Recall that the rank of a flat is the number of non-isolated vertices minus the number of connected components. For each flat in $\F$, the rank of $K_{I_k}$ is $k+1-1=k$. By construction $K_{I_k}\cap\Gamma=T_{I_k}$ and so the rank of each flat is also $k$. \\

\noindent The combinatorial type of this COF is the \emph{caterpillar tree} (see Figure~\ref{fig:Caterpillar Tree}) where one side has ends $i$ and $j$ and the other side has ends $1$ and $i_{n-3}$. This caterpillar tree corresponds to the top-dimensional cone $$\sigma=\textrm{cone}\left(\rho_{I_1}, \rho_{I_2},\ldots, \rho_{I_{n-3}}\right)\in\Mzerontrop$$
Here, $\sigma$ is an example of a cone of $\Mzerontrop$ that doesn't get subdivided when refined to the radially aligned subdivision. Therefore $\projg$ doesn't contract all top-dimensional cones of $\Mzerontrop$ and $\Mzeronrad$.
\end{proof}

\begin{figure}[h]
    \centering
\begin{tikzpicture}
\fill (0,0) circle (0.10);
\draw[very thick] (0,0) -- (-0.5,0.5); 
\node at (-0.8,0.5) {$j$};
\draw[very thick] (0,0) -- (-0.5,-0.5); 
\node at (-0.8,-0.5) {$i$};
\draw[very thick] (0,0) -- (1.2,0);

\fill (1.2,0) circle (0.10);
\draw[very thick] (1.2,0) -- (1.2,0.5); 
\node at (1.2,0.7) {$i_{1}$};
\draw[very thick] (1.2,0) -- (2.4,0);

\fill (2.4,0) circle (0.10);
\draw[very thick] (2.4,0) -- (2.4,0.5); 
\node at (2.4,0.7) {$i_{2}$};
\draw[very thick] (2.4,0) -- (3.2,0);

\fill (3.4,0) circle (0.05); 
\fill (3.6,0) circle (0.05); 
\fill (3.8,0) circle (0.05);

\draw[very thick] (4,0) -- (4.8,0);
\fill (4.8,0) circle (0.10);
\draw[very thick] (4.8,0) -- (4.8,0.5); 
\node at (4.8,0.7) {$i_{n-5}$};
\draw[very thick] (4.8,0) -- (6,0);

\fill (6,0) circle (0.10);
\draw[very thick] (6,0) -- (6,0.5); 
\node at (6,0.7) {$i_{n-4}$};
\draw[very thick] (6,0) -- (7.2,0);

\fill (7.2,0) circle (0.10);
\draw[very thick] (7.2,0) -- (7.7,0.5); 
\node at (8.2,0.5) {$i_{n-3}$};
\draw[very thick] (7.2,0) -- (7.7,-0.5); 
\node at (8,-0.5) {$1$};
\end{tikzpicture}
    \caption{A Caterpillar Tree}
    \label{fig:Caterpillar Tree}
\end{figure}
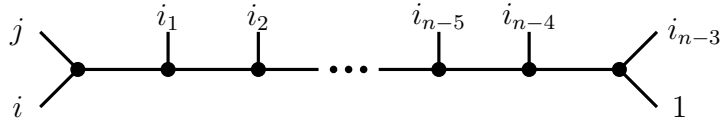


Note that the support of $\BGammaprime$ is a subset of $\RR^{|E(\Gamma)|}/L$ where $L$ is the lineality space spanned by the all ones vector. It is a computation to see that the dimension of $\RR^{|E(K_{n-1})|}/L/S$ is the same as the dimension of $\RR^{|E(\Gamma)|}/L$. There is a natural isomorphism between these two spaces given by underlying matroidal structure. In other words, the standard basis vectors of each space are given by edges in their respective graphs and the vectors in $S$ correspond to precisely the edges not in $\Gamma$. With this in mind we reach the first theorem of this paper.\\

\begin{thm}\label{thm:BergmanFanEqualsProjection}
Let $\Gamma$ be a connected graph on $n-1$ vertices. Then $\projgrad = \BGammaprime$. 
\end{thm}
\begin{proof}
Since $\Gamma$ is a connected graph the dimension of $\BGammaprime$ is $\textrm{rank}(\Gamma)-1$ where $\textrm{rank}(\Gamma)$ is the number of vertices minus 1. In total, we obtain $\dim\left(\BGammaprime\right)=n-3=\dim\left(\Mzeronrad\right)$. 
By Lemma~\ref{lem:ProjectionDoesNotContractAllTopDimCones} the dimension of $\projg(\Mzeronrad)$ is also $n-3$. 
Since $\BGammaprime$ is irreducible and has the same dimension, $\projg$ is surjective.
\end{proof}

\subsubsection{$\BGammaprime$ equals $\MzeroGrad$ for $\Gamma$ complete $k$-partite}

As mentioned earlier, both $\MzeroGtrop$ and $\MzeroGrad$ are well-defined as cone complexes but may not be able to be embedded into $Q_n/U$. Geometrically the issue is that these fans may contain cones which are adjacent to only 1 maximal cell, and thus cannot be balanced. We proceed by taking a combinatorial approach to the problem.

To investigate the relationship between $\MzeroGrad$ and $\BGammaprime$ consider the locus of $\Mzeronrad$ of $\Gamma$-stable curves given by the section $\iota$, which is the natural inclusion map. We define $\Psi_\Gamma:=\projg\circ(\Psi\circ \iota)$ as a map of cone complexes in the following diagram.

\begin{center}
\begin{tikzcd}[row sep=2cm, column sep=2cm,]
\Mzeronrad \arrow[d, bend right=30, "c_\Gamma"'] \arrow[r, "\Psi"', "="] & \BKprime \arrow{d}{\projg} \\
\MzeroGrad \arrow[u, bend right=30, "\iota"'] \arrow[r, "\Psi_\Gamma"', dashed] & \BGammaprime
\end{tikzcd}
\end{center}

The map $\Psi_\Gamma$ induces a map (denoted by the same name) between $\Gamma$-stable radially aligned combinatorial types and COFs of $\Gamma$.
Note that $\Psi\circ\iota$ induces a bijection between the set $\Gamma$-stable radially aligned combinatorial types and $\Gamma$-stable flats of $K_{n-1}$.
Hence statements about $\Psi_\Gamma$ are equivalent to statements about $\projg$ restricted to $\Gamma$-stable flats. 

By showing $\Psi_\Gamma$ is a bijective map between the set of $\Gamma$-stable radially aligned combinatorial types and the set of COFs of $\Gamma$, we obtain an induced bijection of the fans $\MzeroGrad$ and $\BGammaprime$. The next lemma shows that surjectivity of this map is not hard come by and follows from the fact that flats of $\Gamma$ are flats of $K_{n-1}$ restricted to the edge set of $\Gamma$

\begin{lem}
The map $\Psi_\Gamma$ is surjective. 
\end{lem}
\begin{proof}
Consider the COF $\F$ of $\Gamma$ given by $F_1\subset\cdots\subset F_r$ where $F_i$ has $k_i$ connected components. Write the vertex set of each connected component of $F_i$ as $I_j^i$. Construct the COF $\G$ of $K_{n-1}$ as 
$$\G:\coprod_{j=1}^{k_1}K_{I_j^1}\subset\cdots\subset\coprod_{j=1}^{k_r}K_{I_j^r}.$$
Then we have $\projg(\G)=\F$ and thus $\Psi_\Gamma$ is surjective.
\end{proof}

The interesting part of the map $\Psi_\Gamma$ is that it is not always injective. The obstruction is highlighted in the following example.

\begin{exam}\label{exam:ObstructionOfFansOnK4Minus2AdjacentEdges}
Let $\Gamma$ be the subgraph of $K_4$ with edges $e_{35}$ and $e_{45}$ removed, see Figure~\ref{fig:GammaForExampleOfprojgFailingInjectivity}. In $\Mzerontrop$, there are now 8 combinatorial types with 1 bounded edge and 9 combinatorial types with 2 bounded edges that are $\Gamma$-stable. This means that $\M_{0,\Gamma}^{\textrm{trop}}$, as a cone complex, has 8 rays and 9 $2$-dimensional cones and $\M_{0,\Gamma}^{\textrm{rad}}$ has 9 rays and 10 $2$-dimensional cones as described in Example~\ref{exam:Mzero5tropEqualsBK4prime}. 

\begin{figure}[h]
    \centering
\begin{tikzpicture}[scale=2]
    \node (v2) at (1,1.732) [circle, draw = black, inner sep =  2pt, outer sep = 0.5pt, minimum size = 4mm, line width = 1pt] {$2$};
    \node (v3) at (2,0) [circle, draw = black, inner sep =  2pt, outer sep = 0.5pt, minimum size = 4mm, line width = 1pt] {$3$};
    \node (v4) at (0,0) [circle, draw = black, inner sep =  2pt, outer sep = 0.5pt, minimum size = 4mm, line width = 1pt] {$4$};
    \node (v5) at (1,0.577) [circle, draw = black, inner sep =  2pt, outer sep = 0.5pt, minimum size = 4mm, line width = 1pt] {$5$};
    
    \draw[line width = 1pt] (v2) -- node [midway, right] {$e_{23}$} (v3);
    \draw[line width = 1pt] (v2) -- node [midway, left] {$e_{24}$} (v4);
    \draw[line width = 1pt] (v2) -- node [midway, below left] {$e_{25}$} (v5);
    \draw[line width = 1pt] (v3) -- node [midway, below] {$e_{34}$} (v4);
\end{tikzpicture}
    \caption{The graph $\Gamma$ in Example~\ref{exam:ObstructionOfFansOnK4Minus2AdjacentEdges}}
    \label{fig:GammaForExampleOfprojgFailingInjectivity}
\end{figure}
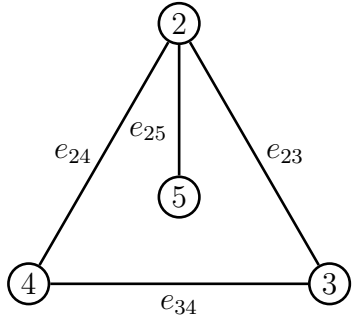

It is important to note that as cone complexes $\BGammaprime$ is not equal to $\M_{0,\Gamma}^{\textrm{trop}}$ nor $\M_{0,\Gamma}^{\textrm{rad}}$, see Figures~\ref{fig:MzeroGammaRadFanObstruction} and \ref{fig:BergmanFanOfGammaObstruction}. The obstruction lies in the ray $\rho=\rho_{\{3,4,5\}}$ and the cone $\sigma=\textrm{cone}(\rho_{\{3,4,5\}}, \rho_{\{3,4\}})$. Let $\C_\rho$ and $\C_\sigma$ be their corresponding combinatorial types. 
Geometrically, $\rho$ is adjacent to only 1 $\Gamma$-stable maximal cell. Meaning there is no way to embed $\rho$ and $\sigma$ into a vector space as a balanced fan.

Write the lattice of flats of $\Gamma$ with the same labels as in Example~\ref{exam:Mzero5tropEqualsBK4prime}
\begin{align*}
    &\textrm{Rank } 1:~ F_1=\{e_{23}\},~ F_2=\{e_{24}\},~ F_3=\{e_{34}\},~ F_6=\{e_{25}\}\\
    &\textrm{Rank } 2 \textrm{ connected}:~ F_7=\{e_{23},e_{24},e_{34}\},~ F_8=\{e_{23},e_{25}\},~ F_{10}=\{e_{24},e_{25}\}\\
    &\textrm{Rank } 2 \textrm{ disconnected}:~ F_{13}=\{e_{25},e_{34}\}
\end{align*}
The flat corresponding to $\rho_{\{3,4,5\}}$ in $K_4$, i.e. $(\Psi\circ\iota)(\C_\rho)$, is $K_{\{3,4,5\}}=F_9$. 
When restricting the edge set, $$K_{\{3,4,5\}}\cap\Gamma=K_{\{3,4\}}=F_3.$$ 
Similarly, $(\Psi\circ\iota)(\C_\sigma)$ is the COF $K_{\{3,4\}}\subset K_{\{3,4,5\}}$ and this COF reduces to the single flat $K_{\{3,4\}}$ when restricting the edge set. 
That is to say there are 3 combinatorial types of $\MzeroGrad$ whose cones all coincide in $\BGammaprime$, namely $$\Psi_\Gamma(\C_\rho)=\Psi_\Gamma(\C_\sigma)=\Psi_\Gamma(\C_{\rho_{\{3,4\}}})=\rho_{F_3}.$$
The map $\Psi_\Gamma$ goes from the cone complex depicted in Figure~\ref{fig:MzeroGammaRadFanObstruction} to the one in Figure~\ref{fig:BergmanFanOfGammaObstruction}. 
Here we can clearly see the obstruction in $\M_{0,\Gamma}^\textrm{rad}$ and how it is collapsed in $\B'(\Gamma)$.

\end{exam}

\begin{figure}[h]
\begin{subfigure}{.48\textwidth}
    \centering
\begin{tikzpicture}[scale=0.8]
    \draw[thick] (18:4cm) -- (90:4cm) -- (162:4cm) -- (234:4cm);
    \draw[thick] (234:2cm) -- (18:2cm) -- (162:2cm) -- (306:2cm);

    \draw[fill=black] (90:0.618cm) circle (3pt) node[label=below:$F_{13}$]  {};

    \foreach \x in {18,162,234}
{
    \draw[thick] (\x:2cm) -- (\x:4cm);
}
    \draw[fill=black] (18:2cm) circle (3pt) node[label=90:$F_{6}$] {};
    \draw[fill=black] (162:2cm) circle (3pt) node[label=270:$F_{3}$] {};
    \draw[fill=black] (234:2cm) circle (3pt) node[label=360:$F_{10}$] {};
    \draw[fill=black] (306:2cm) circle (3pt) node[label=30:$F_{9}$] {};
    \draw[fill=black] (18:4cm) circle (3pt) node[label=18:$F_{8}$] {};
    \draw[fill=black] (90:4cm) circle (3pt) node[label=90:$F_{1}$] {};
    \draw[fill=black] (162:4cm) circle (3pt) node[label=162:$F_{7}$] {};
    \draw[fill=black] (234:4cm) circle (3pt) node[label=234:$F_{2}$] {};
\end{tikzpicture}
    \caption{A slice of the cone complex of $\M_{0,\Gamma}^\textrm{rad}$ with rays labeled by their corresponding flats. When embedded, some angles may be flat.}
    \label{fig:MzeroGammaRadFanObstruction}
\end{subfigure}
\hfill
\begin{subfigure}{.48\textwidth}
    \centering
\begin{tikzpicture}[scale=0.8]
    \draw[thick] (18:4cm) -- (90:4cm) -- (162:4cm) -- (234:4cm);
    \draw[thick] (234:2cm) -- (18:2cm) -- (162:2cm);

    \draw[fill=black] (90:0.618cm) circle (3pt) node[label=below:$F_{13}$] {};

    \foreach \x in {18,162,234}
{
    \draw[thick] (\x:2cm) -- (\x:4cm);
}
    \draw[fill=black] (18:2cm) circle (3pt) node[label=90:$F_{6}$] {};
    \draw[fill=black] (162:2cm) circle (3pt) node[label=270:$F_{3}$] {};
    \draw[fill=black] (234:2cm) circle (3pt) node[label=360:$F_{10}$] {};
    \draw[fill=black] (18:4cm) circle (3pt) node[label=18:$F_{8}$] {};
    \draw[fill=black] (90:4cm) circle (3pt) node[label=90:$F_{1}$] {};
    \draw[fill=black] (162:4cm) circle (3pt) node[label=162:$F_{7}$] {};
    \draw[fill=black] (234:4cm) circle (3pt) node[label=234:$F_{2}$] {};
\end{tikzpicture}
    \caption{A slice of the cone complex of $\B'(\Gamma)$ with rays labeled by their corresponding flats. When embedded, some angles may be flat.}
    \label{fig:BergmanFanOfGammaObstruction}
\end{subfigure}
    \caption{}
\end{figure}

\begin{rem}\label{rem:WeMustRemoveTheCliques!(KeyObservation)}
We saw in Example~\ref{exam:ObstructionOfFansOnK4Minus2AdjacentEdges} that $\projg$ (restricting the edge set) was not injective on cones corresponding to the flats which dropped in rank. 
In particular, the obstruction was a $K_3$ subgraph which had 2 of its 3 edges deleted. 
So we can only allow a graph $\Gamma$ if it has the property that if you delete 2 edges of a $K_3$, then the third edge must also be deleted. 
This is equivalent to the third characterization in Lemma~\ref{lem:AlternateDescriptionsOfCompleteMultipartiteGraphs}.



\end{rem}

\begin{prop}\label{prop:CriterionForInjectivityOfprojg}
The map $\projg$ is injective on cones corresponding to $\Gamma$-stable flats if and only if for any $\Gamma$-stable flat $F$, $\rk(F)=\rk(\projg(F))$.
\end{prop}
\begin{proof}
First let's assume that $\projg$ is injective. That is, for any two distinct $\Gamma$-stable flats, their images under $\projg$ are distinct. Let $F$ be a $\Gamma$-stable flat of $K_{n-1}$. Let $T$ be a spanning forest of $\projg(F)=F\cap\Gamma$. By way of contradiction,  suppose that $\rk(F\cap\Gamma)<\rk(F)$. Consider $\clK(T)$ as a flat of $K_{n-1}$. Then we have
$$\projg(F)=F\cap\Gamma=\clG(T):=\clK(T)\cap\Gamma=\projg(\clK(T)).$$
But since $\rk(\clK(T))=\rk(T)<\rk(F)$, $\clK(T)\neq F$. This contradicts the injectivity of $\projg$ on $\Gamma$-stable flats.\\

Now we will prove the backwards direction. Suppose that for any $\Gamma$-stable flat $F$, then $\rk(F)=\rk(\projg(F))$. Let $F$ and $G$ be $\Gamma$-stable flats of $K_{n-1}$ with $\projg(F)=\projg(G)$. By our hypothesis, we can deduce that $\rk(F)=\rk(G)$. By Lemma~\ref{lem:RankOfASubgraphIFFCommonSpanningForest} $F$ and $G$ share a spanning forest, call it $T$. Then by definition, $\clK(T)=F$ and $\clK(T)=G$. This proves $\projg$ is injective, completing the proof.
\end{proof}

\begin{lem}\label{lem:RankDoesNotDropWhenGammaNice}
Suppose $C$ is a clique of $K_{n-1}$ and that $\Gamma$ is a complete multipartite graph labeled by the same $n-1$ vertices. Then $\rk(C)=\rk(\projg(C))$ or $\rk(\projg(C))=0$.
\end{lem}
\begin{proof}
Suppose that $\rk(\projg(C))\neq0$, i.e. $\projg(C)$ is not the empty graph. Fix an edge $e_{ij}$ between vertices $v_i$ and $v_j$. By Lemma~\ref{lem:AlternateDescriptionsOfCompleteMultipartiteGraphs} for any other vertex $v_k$ either $e_{ik}$ or $e_{jk}$ exists. So there is a path between any two vertices of $\projg(C)$ going through the edge $e_{ij}$. This means that $\projg(C)$ is connected and any spanning tree contains all $n-1$ vertices, proving the lemma.
\end{proof}

\begin{lem}\label{lem:projgInjectiveWhenGammaNice}
The map $\projg$ is injective on cones corresponding to $\Gamma$-stable flats if and only if $\Gamma$ is a complete multipartite graph.
\end{lem}
\begin{proof}
First we will prove the backwards direction. Let $F$ be a $\Gamma$-stable flat of $K_{n-1}$. Note that $F$ is a disjoint union of cliques, $C_i$. By assumption, the image of each clique, under $\projg$, is not empty. Using Lemma~\ref{lem:RankDoesNotDropWhenGammaNice}, we have
$$\rk(F) = \sum_{i=1}^k \rk(C_i) = \sum_{i=1}^k \rk(C_i\cap\Gamma) = \rk(F\cap\Gamma) = \rk(\projg(F)).$$
By Proposition~\ref{prop:CriterionForInjectivityOfprojg}, $\projg$ is injective on cones corresponding to $\Gamma$-stable flats.\\

Now suppose that $\projg$ is injective on cones corresponding to $\Gamma$-stable flats. It is enough to prove the equivalent statement from Lemma~\ref{lem:AlternateDescriptionsOfCompleteMultipartiteGraphs}. Let $v_i$ and $v_j$ be vertices of $\Gamma$ such that $e_{ij}$ is an edge of $\Gamma$. Fix another vertex $v_k$. Consider the flat $F=K_{\{v_i,v_j,v_k\}}.$
We know that 
$$K_{\{v_i,v_j\}}\subset\projg(F).$$
Since $\rk(\projg(F))=\rk(F)=2$, either $e_{ik}$ or $e_{jk}$ must exist as edges in $\Gamma$.
\end{proof}

Geometrically, Lemma~\ref{lem:RankDoesNotDropWhenGammaNice} and Lemma~\ref{lem:projgInjectiveWhenGammaNice} show that when $\Gamma$ is a complete multipartite graph, the cone complex $\MzeroGrad$ will not contain a ray which is adjacent to only 1 maximal cell and thus can be embedded as a balanced fan.

\begin{thm}\label{thm:BergmanFanEqualsMzeroGammarad} 
The cone complex underlying $\MzeroGrad$ is naturally identified with $\projgrad = \BGammaprime$ if and only if $\Gamma$ is a complete multipartite graph. In particular, this complex has the structure of a balanced fan.
\end{thm}
\begin{proof}
By Lemma~\ref{lem:projgInjectiveWhenGammaNice} and Proposition~\ref{prop:CriterionForInjectivityOfprojg}, $\projg$ induces a bijection between the set of $\Gamma$-stable flats of $K_{n-1}$ and flats $\Gamma$ only when $\Gamma$ is a complete multipartite graph. In this case $\Psi_\Gamma$ is a bijection between $\Gamma$-stable radially aligned combinatorial types of $\MzeroGrad$ and flats of $\Gamma$. Thus the map $\Psi_\Gamma$ induces an isomorphism of cone complexes on the ambient vector spaces. We finish the proof with noting that $\MzeroGrad$ is a balanced fan by the fact that it has the same structure as the balanced fan $\projgrad$.
\end{proof}

{
\bibliographystyle{plain}
\bibliography{References.bib}
}
\end{document}